\documentclass[11pt]{article}
\usepackage{amssymb,amsthm,amsfonts,hhline,color}

\def\C{{\bf C}}

\def\F{{\bf F}}
\def\Z{{\bf Z}}

\def\x{{\bf x}}
\def\y{{\bf y}}
\def\o{{\omega}}
\def\bo{{\bar\omega}}

\def\halb{\hbox{$\frac{1}{2}$}}

\title{A new class of codes over $\Z_2\times \Z_2$}

\author{Julia Galstad and Gerald H\"ohn}

\date{August 2010}

\begin{document}

\bibliographystyle{amsalpha}

\theoremstyle{plain}
\newtheorem{thm}{Theorem}[section]
\newtheorem{prop}[thm]{Proposition}
\newtheorem{lem}[thm]{Lemma}
\newtheorem{cor}[thm]{Corollary}
\newtheorem{rem}[thm]{Remark}
\newtheorem{conj}[thm]{Conjecture}

\theoremstyle{definition}
\newtheorem{defi}[thm]{Definition}

\renewcommand{\baselinestretch}{1.2}

\maketitle

\section{Introduction}\label{intro}

In this paper, we discuss a class of codes which we call $L$-codes. They arise
naturally as a fifth step in a series of analogies between Kleinian codes,
binary codes, lattices and vertex operator algebras (see~\cite{Ho-kleinian,CoSl,Ho-dr} for
a discussion of these analogies).
In many aspects, $L$-codes are similar to Kleinian codes \cite{CRSS-quant,Ho-kleinian}.
In particular, both are additive codes over the Kleinian four group. However, it is more natural
to see them as part of the above sequence. Then results known for the other
four cases, in particular for self-dual objects, generalize naturally to $L$-codes.

\smallskip

An important part of our viewpoint of $L$-codes is the relation between $L$-codes, Kleinian and
binary codes of various lengths. We discuss several maps in detail.
The map $\psi$, which is conceptually the most important one, identifies
$L$-codes of length~$n$ with subsets of $(\delta_2^{\bot}/\delta_2)^n$, where
$\delta_2$ is the Kleinian code $\{(0,\,0),\,(a,\,a)\}$ (see Section 3) and is motivated by
the analogy mentioned at the beginning. The map $\phi$ essentially considers an $L$-codes as a Kleinian code by forgetting some structure.
Although not very natural, this provides the most efficient computational
approach to classification of $L$-codes, in particular self-dual ones.

Many proofs are easy and standard in coding theory and so we often skip details.
Sometimes we provide complete proofs since we like to emphasize the language of $L$-codes,
although the results would directly follow from corresponding results for Kleinian
or binary codes. However, if it is more convenient to deduce results from
Kleinian or binary codes, we do so; an example is the mass formula for self-dual codes.

\medskip

The paper is organized as follows. In the rest of the introduction our motivation
for introducing $L$-codes is explained. These codes arise quite naturally from various
viewpoints if one studies vertex operator algebras.
In Section~\ref{basic}, we introduce the notion of
$L$-codes and give related definitions, like scalar products and weight enumerators.
We also prove some basic theorems. In Section~\ref{relations}, we discuss the above mentioned
relation of $L$-codes to other type of codes. In the final section, we study the classification
of self-dual codes. All self-dual codes up to length $10$ will be classified. We provide
explicit tables up to length~$4$. We also discuss some results about extremal codes
for which details will be provided in a further paper.  Most results of this paper are from
the second author's master's thesis~\cite{Ga-master}.

\bigskip

\paragraph{Motivations coming from vertex operator algebra theory.}
We describe three different motivations to introduce the notion of $L$-codes arising from
the viewpoint of vertex operator algebras.

\medskip

In general, it is a difficult problem to describe the possible extensions of a rational
vertex operator algebra. The isomorphism classes of extensions are described completely
in terms of the associated {\it modular tensor category\/} (essentially the same as a three-dimensional
quantum field theory); cf.~\cite{Ho-genus}. One idea to construct new vertex operator algebras
(or to give different description of known ones) is to take the tensor product of one
fixed vertex operator algebra and to study the extension of it. Then one has to develop
a ``coding theory'' for the modular tensor category of the vertex operator algebra one starts with.
For example, extensions of tensor products of the lattice vertex operator algebra for the root lattice $A_1$
are described by the well-studied doubly-even self-orthogonal codes. For the Virasoro
vertex operator algebra of central charge~$1/2$, this leads to the theory of framed vertex operator
algebras~\cite{DGH-virs}.
For the lattice vertex operator algebra for the root lattice $D_4$, one obtains what was
called Kleinian codes in~\cite{Ho-kleinian}. The corresponding modular tensor categories have in
these three cases not more than four simple objects. Such tensor categories can be classified,
cf.~\cite{Eh-dr,RSW-modular}. Inspecting the list of these tensor categories (cf.~the list in the
database~\cite{HGY-database}) one sees that there are essentially five cases for which
the tensor category is described by a quadratic form on an abelian group. In these
cases one ends up with a ``classical coding'' theory as studied in quite generality ~\cite{NSC-book}.
They are binary codes, ternary codes, $\Z_2\times \Z_2$-codes in the form of Kleinian codes, two related cases
of codes over $\Z_4$ and one other type of codes over $\Z_2\times \Z_2$ which we will call $L$-codes in
this paper. Kleinian and $L$-codes and codes share the same scalar product which make them quite similar, but the
corresponding quadratic forms are different and thus the notion of {\it even codes\/} will be different.

In addition, the concept of {\it Euclidean weight\/} is an additional structure not coming alone from the
modular tensor category, but naturally described in terms of a vertex operator algebra
realizing the modular tensor category in question. In the case of $L$-codes, such a
vertex operator algebra must have central charge divisible by $8$. For the smallest
possible case of central charge there is exactly one such vertex operator algebra, namely
the vertex operator algebra associated to the root lattice $D_8$. (This can for example be
seen in the following way: The vertex operator algebra corresponding to the even self-dual $L$-code $\Xi_1$
is self-dual and it is known that the unique self-dual vertex operator algebra of central
charge~$8$ is the one associated to the $E_8$-lattice. A vertex operator subalgebra corresponding
to the $L$-code $\{0\}$ corresponds to an involution in the exceptional Lie group $E_8(\C)$
which is the automorphism group of the $E_8$-lattice vertex operator algebra. There are only
two conjugacy classes of involutions in $E_8(\C)$ and the one corresponding to the code $\{0\}$
must be the one of type $2B$.) The four isomorphism classes of irreducible modules of the $D_8$ lattice vertex operator
algebra have the conformal weights~$0$, $1/2$, $2$ and $2$. This explains our choice of the
Euclidean weight of $L$-codes. In summary, even self-orthogonal $L$-codes describe extensions of
tensor products of the $D_8$ lattice vertex operator algebra.

\smallskip

A further reason to consider $L$-codes is that they arise naturally as another step in
the analogy between Kleinian codes, binary codes, lattices and vertex operator algebras
as discussed in~\cite{Ho-kleinian}, Section~7. As mentioned loc.~cit.~there is one more
such step, and these are $L$-codes. We will describe two kind of questions for which
this viewpoint may be helpful.

\smallskip

One application are the study of {\it automorphism groups\/} of vertex operator algebras. The automorphism group
of the fixed point lattice vertex operator algebra $V_L^+$ for an lift of the $(-1)$-isometry
of a lattice $L$ was described in general by Shimakura~\cite{Shimakura1, Shimakura2}. In particular, the description
allows to compute its order (cf.~the corresponding entries in the VOA database~\cite{HGY-database}).
However, it is difficult to determine the exact shape of ${\rm Aut}(V_L^+)$ in general. Complications
arise if the lattice $L$ arises from a binary code~$C$ and then further if $C$ arises from
a Kleinian code $D$. This continues if the code $D$ comes from an $L$-code. Thus the automorphism groups
of $L$-codes play a role in the description of the automorphism group of certain naturally arising
vertex operator algebras.

An important example of a vertex operator algebra is the Moonshine module~$V^\natural$.
As explained in~\cite{DGH-virs}, the Moonshine module can be constructed naturally by twisted constructions
from what in the notation of the present paper is the $L$-code~$\Upsilon_3$. However,
there is currently no purely vertex operator algebraic description of the automorphism group
of $V^\natural$, the sporadic simple group known as the Monster. One hope is that a
similar analysis for the twisted construction as was done for the $+$-construction by~\cite{Shimakura1, Shimakura2}
may finally provide a vertex operator algebraic understanding of the Monster by starting from $\Upsilon_3$.

\smallskip

The second application we have in mind is a better understanding of {\it extremal vertex operator algebras.\/}
Extremal vertex operator algebras have been introduced by the second author in~\cite{Ho-dr} as an analogue of extremal
binary codes and extremal lattices. Codes and lattices are used in information science to transmit
information error-free over noisy channels, cf.~\cite{Shannon}. With the theoretical development
of quantum computing and quantum information theory, cf.~\cite{Shor-capacity}, Kleinian codes have been used as
quantum codes~\cite{CRSS-quant}. More realistic models of computing and information must be based on quantum field theory.
The well-known difficulties in understanding realistic nonperturbative models of quantum field theory in four dimensions
motivated the investigation of simpler three-dimensional models. In particular, it was shown that
a three-dimensional quantum field theory computer may used for $\#P$-complete problems~\cite{Freedman-computer}.
The ultimate limit for the information capacity is given by the semi-classical Bekenstein-Hawking formula
of the entropy of a black hole in terms of the area of its event horizon~\cite{Hawking-rumms}. For the
three-dimensional case it was suggested more recently by Witten~\cite{witten-3dgravity}, that extremal vertex operator algebras of
central charge a multiple of~$24$ may provide a precise quantum field theoretical description of three-dimensional black holes via
the AdS/CFT correspondence. Unfortunately, there is no such vertex operator algebra known besides the
Moonshine module, but there is also no known limit on the size central charge as is the case
for lattices, binary codes and Kleinian codes. However, even in the case of lattices, binary codes and Kleinian codes,
the known bounds are much larger than the largest constructed examples.
The situation for $L$-codes is quite different: restricting to the cases where the length is a multiple of $3$
(which corresponds to vertex operator algebras of central charge a multiple of~$24$) we find that
even extremal $L$-codes exist only for the length $3$ (the code $\Upsilon_3$) and length~$6$ (one code derived
from the Hexacode).  This gives some numerical evidence that there should exist extremal vertex operator
algebras at least for the central charge~$48$ if not for larger central charge. Note that there is also a unique even extremal Kleinian code of length~$12$~\cite{CRSS-quant}, a unique even extremal code of length~$48$~\cite{self-dual-48},
and at least three even extremal lattices of rank~$48$~\cite{CoSl}.


\section{Definitions and examples}\label{basic}

\subsection{$L$-codes}
We let $L$ be the Kleinian four group $\Z_2^2$ and denote its four elements by $0$, $1$,
$\omega$ and $\bar\omega$. We also consider the map
$|\,.\,|^2: L\longrightarrow \Z$ defined by $|0|^2=0$, $|1|^2=1$ and $|\o|^2=|\bo|^2=2$
and the map $q:L\longrightarrow \F_2$, $q(x)=|x|^2\bmod{2}$, i.e.~$q(0)=q(\o)=q(\bo)=0$ and $q(1)=1$.
Then $q$ is a finite quadratic form and $(L,q)$ becomes a finite quadratic space.
Adding this structure to the Kleinian four group makes $1$ the
distinguished non-zero element of $L$ with the only non-zero value of $q$.
From the quadratic form $q$ we derive a
symmetric, biadditive dot product $\cdot:L\times L\longrightarrow \F_2$,
$x \cdot y = q(x+y)-q(x)-q(y)$.
Explicitly one has
$$\begin{array}{c|cccc}
\cdot & 0 & 1 & \omega & \bar{\omega} \\
\hline
0 & 0 & 0 & 0 & 0 \\
1 & 0 & 0 & 1 & 1 \\
\omega & 0 & 1 & 0 & 1 \\
\bar{\omega} & 0 & 1 & 1 & 0 \\
\end{array}_. $$

We can naturally extend these structures to the $n$-fold direct sum $L^n$ of $L$.
More explicitly, the {\it (Euclidean) weight\/} is the norm extended
to~$L^n$:
$$\mbox{ewt}:L^n \to {\bf Z}, \qquad {\bf x}\mapsto \sum_{i=1}^n {|\,{x_i}|^2},$$
where we write $\x \in L^n$ as a vector ${\bf x} = (x_1,\dots,x_n)$, $x_i \in L$.
The quadratic form extended to $L^n$ is:
$$q:L^n \to {\bf F}_2, \qquad {\bf x} \mapsto \sum_{i=1}^n q(x_i).$$
Clearly, $ q({\bf x})=\mbox{ewt}({\bf x}) \bmod{2}$ for all ${\bf x}\in L^n$.
The associated bilinear form of the quadratic form $q$ on $L^n$ we call
the {\it scalar product\/} and it
is given by
$$(\,.\,,\,.\,):L^n \times L^n \to {\bf F}_2, \qquad ({\bf x},{\bf y})=\sum_{i=1}^n x_i \cdot y_i,$$
where $\cdot$ is the dot product on $L$ as above.

The semidirect product $G=S_2^n{:}S_n$, the wreath product of
$S_2$ by $S_n$, acts on~$L^n$: An element of $G$ consists of a
permutation of the $n$ coordinates together with a permutation of
$\omega$ and $\bar{\omega}$ at each position. This action preserves
all the additional structures of Hamming weight (defined below), Euclidean weight,
quadratic form and scalar product.

\medskip

A {\it code\/} over $L$ of length $n$, or {\it $L$-code\/} for short, is any subset
$C\subset L^n$. Its elements are called {\it codewords.\/} The code $C$ is
called {\it linear\/} if it is a subgroup of $L^n \cong {\bf Z}_2^{2n}$. A
linear code has $4^k$ elements, with $k\in \frac{1}{2}{\bf Z}$. We
call $k$ the {\it dimension\/} of the code.  All the codes we are
considering in this paper are assumed to be linear.
A $subcode$ of {\it C\/} is a subgroup $D$ of $C$.

As usual in coding theory, we let the {\it Hamming weight\/} $\mbox{wt}(\bf{x})$ of a
vector ${\bf x}\in L^n$ be the number of nonzero components~$x_i$.
We will use the concepts of both {\it minimal Hamming weight\/} and
{\it minimal Euclidean weight\/}. The former is defined by
$h(C)=\min\{\mbox{wt}({\bf x}) \mid {\bf x}\in C \backslash \{0\}\}.$
Similarly, the later is
$d(C)=\min\{\mbox{ewt}({\bf x}) \mid {\bf x}\in C \backslash \{0\}\}$.
An $[n,k]$-code is a code of length $n$ and dimension~$k$.
An $[n,k,d]$-code is an $[n,k]$-code of minimal Euclidean weight $d$.

The {\it automorphism group\/} of a code $C$ is the subgroup of $G$ sending $C$ to itself:
$$\mbox{Aut}(C)=\{g \in S_2^n{:}S_n \mid g  C = C\}$$Two
codes $C$ and $D$ are {\it equivalent\/} if there is a $g\in G$ with $g C = D$.
%
It is clear that the number of codes equivalent to a code $C$ is
$\textstyle\frac{2^n \cdot n!}{|{\rm Aut}(C)|}$.

\smallskip

Given an $[n,k]$-code $C$ and an $[m,l]$-code $D$, we can define the
direct sum $C \oplus D$ as the direct product subgroup of $L^n
\oplus L^m$. The direct sum of $C$ and $D$ is an $[n+m,k+l]$ code.
For short, we write $C^l:=\bigoplus_{i=1}^l C$.
If $C = C_1 \oplus
C_2$, for nontrivial codes $C_1$ and $C_2$, then $C$ is called
{\it decomposable,\/} otherwise $C$ is {\it indecomposable.\/} Every code $C$
can be expressed uniquely as a direct sum of indecomposable codes,
up to reordering of the components and equivalence of codes.


\subsection{Weight enumerators}

We define several weight enumerator polynomials for an $L$-code $C\subset L^n$.

The {\it complete weight enumerator\/} is the polynomial
$$\mbox{cwe}_C(p,q,r,s)=\sum_{i,j,k,l}A_{i,j,k,l}\,p^i q^j r^k s^l,$$
where $A_{i,j,k,l}$ is the number of code words in $C$
containing at $i,$ $j,$ $k,$ resp.~$l$ of the $n$ coordinates the
element $0$, $1$, $\omega$, resp. $\bar{\omega}$.
The {\it symmetric weight enumerator\/} does not distinguish
between $\omega$ and $\bar{\omega}$. It is defined to be the degree
$n$ polynomial
$$\mbox{swe}_C(x,y,z)=\sum_{i,j} S_{i,j}\, x^{n-i-j} y^i z^j,$$
where $S_{i,\, j}$ is the number of code words in $C$
containing at~$i$ of the $n$ coordinates the element~$1$
and containing at $j$ of the $n$ coordinates the element
$\omega$ or $\bar{\omega}$.
The {\it (Hamming) weight enumerator \/} of $C$
is the degree $n$ polynomial
 $$\mbox{W}_C(u,v)=\sum_{i=0}^n H_i \,u^{n-i}v^i \hskip 5 mm
\mbox{ with }H_i:=\#\{ {\bf x} \in C \mid \hbox{wt}({\bf x})=i \}.$$
Finally, the {\it Euclidean weight enumerator\/} of
$C$ is the degree $2n$ polynomial
$$\mbox{EW}_C(a,b)=\sum_{i=0}^{2n} E_i \,a^{2n-i}b^i \hskip 5 mm
\mbox{ with }E_i:=\#\{ {\bf x} \in C \mid {\rm ewt}({\bf x})=i \}.
$$

\begin{rem} \label{enumRel} The following weight enumerator relationships hold:
$${\rm swe}_C(x,y,z) = {\rm cwe}_C(x,y,z,z),$$
$${\rm W}_C(u,v) = {\rm swe}_C(u,v,v) = {\rm cwe}_C(u,v,v,v),$$
$${\rm EW}_C(a,b)={\rm swe}_C(a^2,ab,b^2)={\rm cwe}_C(a^2,ab,b^2,b^2).$$
Furthermore, equivalent codes will have the same Hamming, symmetric
and Euclidean weight enumerators. \end{rem}


\subsection{Self-dual codes}

The codes we are primarily interested in are self-dual codes.
There are two concepts of even. A code $C$ is {\it Hamming even\/} if
for any $\x \in C$, $\mbox{wt}(\x)\in 2{\bf Z}$.  A code $C$ is
{\it (Euclidean) even\/} if $\mbox{ewt}(\x)\in 2{\bf Z}$ for all $\x \in C$.
Unless stated otherwise, even will mean Euclidean even. Clearly, a
code is even if and only if $q(\x)=0$ for all $\x \in C$.

The {\it dual code\/} of $C$ is defined as
$$C^{\bot}=\{\x\in L^n \mid (\x,\y)=0
\mbox{ for all } \y\in C \},$$ where $(\,.\,,\,.\,)$ is the scalar
product.
A code $C$ is called \emph{self-orthogonal} if $C\subset
C^{\bot}$, and \emph{self-dual} if $C=C^{\bot}$.

\begin{rem}\label{selfdual1} If $C$ is a code, then $C^{\bot}$ is a linear code. \end{rem}

\begin{lem}\label{dimensionSelfDual2} If $C$ is an $L$-code of type $[n,k]$, then $C^{\bot}$ is an $L$-code of type $[n,n-k]$.
In particular, the dimension of a self-dual code $C$ is
$\frac{n}{2}.$
\end{lem}
\begin{proof}
By Remark~\ref{basicphi} below, the scalar product for
Kleinian codes is the same as for $L$-codes. Hence, the dimension of
$C^{\bot}$ equals $n-k$ when the dimension of $C$ is~$k$, because
this result is true for Kleinian codes \cite{Ho-kleinian}.
\end{proof}
Together with Remark~\ref{selfdual1} this implies:
\begin{rem}\label{selfdual2} If $C$ is a linear code, then ${(C^{\bot})}^{\bot}=C$.
\end{rem}


\begin{lem}\label{selfdual3}
Let $S\subset L^n$ be a set of pairwise orthogonal vectors.
Then the code
$${\rm span}(S)=\left\{ \sum_{i=1}^m x_i \,\big|\, x_i \in S,\,m \in \bf{N} \right\}$$
spanned by $S$ is self-orthogonal.
If in addition the vectors in $S$ are even, then ${\rm span}(S)$ is even.
\end{lem}
\begin{proof}Let $C={\rm span}(S)$. Consider $x$, $y \in C$. Then
$x={\sum_{i=1}^m c_i}$ and $y={\sum_{j=1}^l d_j}$ for some $c_i$, $d_j\in S$, $m$, $l \in \bf{N}$.
One gets $(x,y)=\sum_{i=1}^m\sum_{j=1}^l (c_i,d_j)= 0$.
Hence, $C\subset C^{\bot}$, so $C$ is self-orthogonal.

To prove the second part of the lemma, assume that the vectors in
$S$ are even and pairwise orthogonal.  We need to show that $q(x)=0$
for all $x=\sum_{i=1}^m c_i\in C$.
We proceed by induction on $m\in{\bf N}$.
For $m=0$, we have $q(0)=0$.
Assume that the statement is true for the sum of $k$ elements in $S$ and let $m=k+1$.
Since $C$ is self-orthogonal one has $(c_{k+1},\sum_{i=1}^k c_i)=0$.
By this and the definition of scalar product, we see that
$$q(c_{k+1} + \sum_{i=1}^k c_i)= q(c_{k+1}) + q(\sum_{i=1}^k c_i).$$
Using the inductive assumption and $q(c_{k+1})=0$, the induction step is finished.
\end{proof}


\subsection{Examples of $L$-codes}\label{exOfLcodes}

In the future, when we refer to a {\it code,\/} we sometimes mean an
\textit{equivalence class of codes.}


Some basic examples of self-orthogonal $L$-codes together with a generating set, their basic properties
and their symmetric and Euclidean weight enumerators are given in Table~\ref{examplesLcodes}.
The subscripts on these codes indicate the code length.
Note that the codewords of $\Delta_l$ are all words with an even
number of coordinates $1$ and remaining coordinates~$0$.
The codes $\Upsilon_2$ and $\Upsilon_3$ consist of the sets
$\{(0,0),\,(\omega,1),\,(1,\omega),\,(\bar{\omega},\bar{\omega})\}$
and $\{(0,0,0),\,(1,1,\omega),\,(1,\omega,1),\,(\omega,1,1),$
$(\bar{\omega},\bar{\omega},0),\, (\bar{\omega},0,\bar{\omega}),$
$(0,\bar{\omega},\bar{\omega}),\,(\omega,\omega,\omega)\}$,
respectively.

\begin{table}\caption{Examples of $L$-codes}\label{examplesLcodes}\small
$$\begin{array}{clccccc}
\renewcommand{\arraystretch}{1.3}
C & \hbox{generators} & \hbox{dim} &  {\rm type} & {\rm Aut}(C) & {\rm swe}_C & {\rm EW}_C \\
\hline
\Gamma_1 & \begin{array}{l}(1)\end{array} & \halb &   {\rm N} & S_2     & x+y & a^2+ab  \\
\Xi_1    & \begin{array}{l}(\omega)\end{array} & \halb &  {\rm Y} & \{{\rm id}\}      & x+z & a^2+b^2  \\[2mm]
\Delta_l &  {\renewcommand{\arraystretch}{0.8}\begin{array}{l}
 (1,\,1,\, 0,\,\,\ldots,\, 0),\\ (0,\, 1,\, 1,\, 0,\, \ldots,\, 0), \\\cdots, \\ (0,\,\ldots,\,0,\,1,\ 1 )\\ \end{array} }  &
       \frac{l-2}{2} & {\rm Y} & {\rm S}_2^l{:}S_l & \frac{1}{2}[(x+y)^l+(x-y)^l] & \frac{1}{2}a^l[(a+b)^l+(a-b)^l]  \\[9mm]
\Upsilon_2 & {\renewcommand{\arraystretch}{0.8}\begin{array}{l}
(1,\, \omega), \\ (\omega,\, 1) \end{array} } & 1      & {\rm  N} & S_2         & x^2+2yz+2z^2 & a^4+2ab^3+2b^4  \\[4mm]
\Upsilon_3 &  {\renewcommand{\arraystretch}{0.8}\begin{array}{l}
(1,\,1,\,\omega), \\ (1,\,\omega,\,1), \\ (\omega,\,1,\,1) \end{array}}   & \frac{3}{2} & {\rm Y} & S_3  & x^3+3xz^2+ 3y^2z+z^3 & a^6+6a^2b^4+b^6
\end{array}$$
\end{table}


\section{Relation to Kleinian codes and binary codes}\label{relations}

The theory of $L$-codes is closely related to the theory of Kleinian
codes.  Some useful maps between $L$-codes, Kleinian codes and binary codes
will be defined.

\smallskip

We recall some basic definitions and notations for Kleinian codes from~\cite{Ho-kleinian}.
We let $K$ again be the Kleinian four group and denote the elements of
$K$ by 0, $a$, $b$, and~$c$. The essential difference to $L$-codes is that
we consider another quadratic form on $K$, namely $q:K\longrightarrow \F_2$,
$q(0)=0$ and $q(a)=q(b)=q(c)=1$. The quadratic form is now symmetric with respect to the
whole automorphism group $S_3$ of $K$. However, the induced biadditive form
$\cdot:K\times K\longrightarrow \F_2$, $x \cdot y = q(x+y)-q(x)-q(y)$ is isomorphic
with the one on $L$.

Codes over $K$ are defined analogous to $L$-codes.
We call these codes Kleinian codes as in~\cite{Ho-kleinian}. 
All Kleinian codes are assumed to be linear.
Hamming weight and minimal Hamming
weight are also defined the same way as their $L$-code analogues.
Since we do not define the Euclidean weight for Kleinian codes, we denote by
$d(C)$ the minimal Hamming weight of a Kleinian code $C$.
Kleinian codes are called {\it even\/} if the (Hamming) weights of all codewords
are divisible by~$2$. Since also the scalar products on
$K^n$ and $L^n$ are isomorphic, the notations of self-orthogonal and
self-dual codes coincide.
Consider the automorphisms acting on $K^n$ that form the semidirect product
$H=S_3^n{:}S_n$, consisting of the permutation of the coordinates
together with a permutation of the symbols $a$, $b$ and $c$ at each
position. The automorphism of a Kleinian code $C$ is the subgroup of
$H$ sending $C$ to itself. Two Kleinian codes $C$ and $D$ are
equivalent if there is a $g\in H$ such that $g\,C=D$.
Usually, we will denote specific Kleinian codes by small Greek letters.


\subsection{The map $\phi$}\label{sectionphi}

In this section, we consider a map between $L$-codes and Kleinian codes of the same length.
\begin{defi}
The map $\phi:L^n\rightarrow K^n$ is defined coordinate-wise:
$$(c_1,\,\ldots,\,c_n)\mapsto (d_{1},\,\ldots,\,d_n)$$
where $c_i\mapsto d_i$ is determined by
$0\mapsto 0$, $1\mapsto a$, $\omega\mapsto b$ and $\bar{\omega}\mapsto c$.
\end{defi}

The following remarks follow from what was said before about Kleinian codes.
\begin{rem} \label{basicphi}The map $\phi$ is a bijective homomorphism that preserves the scalar product.
\end{rem}
If $C\subset L^n$ is a (linear) $L$-code, then $\phi(C)$ is a
(linear) Kleinian code, so one can consider
$\phi$ to be a map on codes.
\begin{rem}\label{Wphi} For an $L$-code $C$, $${\rm W}_C(u,v)={\rm W}_{\phi(C)}(u,v).$$
\end{rem}
\begin{rem} \label{selfdualphi} The map $\phi$ sends self-orthogonal (self-dual) codes
in $L^n$ to self-orthogonal (self-dual) codes in $K^n$.  The map
$\phi^{-1}$ behaves similarly.
\end{rem}
Note that if two $L$-codes $C$ and $D$ are equivalent in the sense of $L$-codes,
then $\phi(C)$ and $\phi(D)$ are equivalent as Kleinian codes, but not vice versa.

\begin{cor} The map $\phi$ maps equivalence classes of
self-orthogonal (self-dual) codes in $L^n$ to self-orthogonal
(self-dual) codes in $K^n$.
\end{cor}
We write $[E]$ for the equivalence class of an $L$-code resp.~Kleinian code $E$.
Let $D$ now be a Kleinian code. Then there may be several equivalence classes of $L$-codes
which are mapped by $\phi$ to the same equivalence class $[D]$.
More precisely, we have:
\begin{lem} \label{Phiorbitformula}
For a fixed Kleinian code $D$ one has
$$\frac{|H|}{|{\rm Aut}(D)|}=\sum_{\phi([C])\in [D]}\frac{|G|}{|{\rm Aut}(C)|},$$
where $[C]$ runs though all equivalence classes of $L$-codes which are mapped by~$\phi$
to~$[D]$.
\end{lem}
Note that the order of the automorphism group of a code depends only on its equivalence class.
\begin{proof}The $H$-orbit $[D]$ of $D$ splits into $G$-orbits under the $G$-action
induced by $\phi$. The orbit formula for group actions now implies the result.
\end{proof}

To investigate the different $L$-code equivalence classes corresponding to a given
Kleinian code~$D$, it is useful to introduce the concept of a
{\it marking\/} of a Kleinian code (cf.~Section~7 of \cite{Ho-kleinian}).

\begin{defi}{\rm A {\it marking\/} for a Kleinian code $D\subset K^n$ is the choice of a vector
${\mathcal M }\in(K\setminus\{0\})^n$. The automorphism group ${\rm Aut}_D({\mathcal M})$ of a
marking ${\mathcal M}$ with respect to $D$ are the elements of
${\rm Aut}(D)$ that leave $\mathcal{M}$ fixed. Two markings
$\mathcal{M}_1$ and $\mathcal{M}_2$ are called equivalent if there exists
an $g\in{\rm Aut}_D(\mathcal{M})$ such that
$g\,\mathcal{M}_1=\mathcal{M}_2$.}
\end{defi}

\begin{lem}\label{countMarkings}
Let $D$ be a Kleinian code of length $n$ and $\mathcal{M}$ be a marking of $D$.
Denote the set of markings of $D$ by $\mathfrak{M}$, and the
${\rm Aut}(D)$-orbit of $\mathcal{M}$ in $\mathfrak{M}$ by
$[\mathcal{M}]$. Then for the number of markings of $D$ one has
$$|\,\mathfrak{M}|=3^n=\sum_{[\mathcal{M}]}\frac{|{\rm Aut}(D)|}{|{\rm Aut}_D(\mathcal{M})|},$$
where the sum is over equivalence classes of markings.
\end{lem}
\begin{proof} This is an application of the orbit formula for the group ${\rm Aut}(D)$
acting on the set of markings.
\end{proof}

\begin{lem}\label{one2one} The equivalence classes of markings of a
Kleinian code $D$ are in a one-to-one correspondence with the equivalence
classes of $L$-codes $C$ with $[\phi(C)]=[D]$.
\end{lem}
\begin{proof} Both sets can be identified with double cosets ${\rm Aut}(D) \backslash H/G$ as follows:
The $H$-orbit of $D$ is the set $[D]$ of all Kleinian codes equivalent to $D$ and the stabilizer of
$D$ is ${\rm Aut}(D)$.  This identifies $[D]$ with ${\rm Aut}(D) \backslash H$. The equivalence classes of $L$-codes
are then identified with the orbits of the natural $G$-action on ${\rm Aut}(D) \backslash H$,
cf.~Lemma~\ref{Phiorbitformula}.

Consider the marking $\mathcal{M}=(a,\,a,\,\ldots,\,a)$. Then the $H$-orbit of $\mathcal{M}$
is the set $\mathfrak{M}$ of all markings and the stabilizer of $\mathcal{M}$ is $G$. This
identifies $\mathfrak{M}$ with $H/G$. The equivalence classes of markings are then identified
with the orbits of the natural ${\rm Aut}(D$)-action on $H/G$, cf.~Lemma~\ref{countMarkings}.
\end{proof}

Given a marking $\mathcal{M}$ for a Kleinian code $D$, an $L$-code $C$ with $[\phi(C)]=[D]$ and
belonging to the $L$-equivalence class $[C]$ corresponding to the ${\rm Aut}(D)$-class $[\mathcal{M}]$ can be constructed
by the map $\phi_\mathcal{M}^{-1}:=\phi^{-1}\circ \mu_\mathcal{M}: K^n\longrightarrow L^n$, where
$\mu_\mathcal{M}:K^n\longrightarrow K^n$, $\x\mapsto \mu_\mathcal{M}(\x)$, is defined for each coordinate $i=1$, $\ldots$, $n$ as the
product~$\mathcal{M}_i\cdot x_i$ and $K$ is identified with the field $\F_4$ of four elements
such that $a\in K$ is the multiplicative identity. Then $C=\phi_\mathcal{M}^{-1}(D)$.
Note that $\phi_\mathcal{M}^{-1}=\phi^{-1}$ for the {\it standard marking\/} $\mathcal{M}=(a,\,\ldots,\,a)$.

\begin{rem} The maps $\phi$ and $\phi_{\mathcal{M}}^{-1}$ do not necessarily map even codes to even codes.
\end{rem}
For example, for the even $L$-code $\Xi_1$ we have $\phi(\Xi_1)=\{(0),(b)\}\cong \gamma_1$, which is
an odd Kleinian code, and for the even Kleinian code $\epsilon_2$ we have $\phi_{(a,b)}^{-1}(\epsilon_2)=\Upsilon_2$,
which is an odd $L$-code.

\medskip

\noindent{\bf Example:\/} We illustrate the above results for the Hexacode ${\cal C}_6$ considered
as a Kleinian code. The orbits of ${\rm Aut}({\cal C}_6)$ on $K^6$ have been determined in
Table~4~\cite{Ho-kleinian}. Since a marking for a Kleinian code of length $n$ is the same as a vector
in $K^n$ of weight~$n$, we can read off from this table that there are exactly five equivalence
classes of markings for ${\cal C}_6$ having size $18$,  $180$, $216$, $45$ and $270$.
Table~\ref{HexaLcodes} lists the order of the automorphism group and the coefficients of the Euclidean weight enumerator
of the corresponding $L$-codes.
The $L$-code no.~1 has minimal Euclidean weight~$6$ and we have chosen the generators
$(\bo,   \o,   \o,   0,   0,   \o)$,
$(  \o, \bo,   0,   \o,   0,   \o)$,
$(  \o,   0, \bo,   0,   \o,   \o)$,
$(  0,   \o,   0, \bo,   \o,   \o)$,
$(  0,   0,   \o,   \o, \bo,   \o)$ and
$(  \o,   \o,   \o,   \o,   \o, \bo)$.

\begin{table}\caption{$L$-codes mapped by $\phi$ to the Hexacode ${\cal C}_6$}\label{HexaLcodes}\small
$$\begin{array}{ccc|rrrrrrrrrrrrr}
{\rm No.\/}  & \!\!|{\rm Aut}(C)|\!\! & \mathcal{M} &   E_0 & E_1 & E_2 & E_3 & E_4 & E_5 & E_6 & E_7 & E_8 & E_9 & E_{10} & E_{11} & E_{12}  \\ \hline
1 &  120 & (aaaaaa) & 1 & 0 & 0 & 0 & 0 & 0 & 31 &  0 & 15 & 0 & 15 & 0 & 2 \\
2 &   12 & (abccba) & 1 & 0 & 0 & 0 & 0 & 6 & 12 & 18 & 12 & 2 &  6 & 6 & 1 \\
3 &   10 & (baacbc) & 1 & 0 & 0 & 0 & 0 & 5 & 15 & 16 & 10 & 5 &  5 & 6 & 1 \\
4 &   48 & (bccaac) & 1 & 0 & 0 & 0 & 3 & 0 & 12 & 24 &  6 & 8 &  6 & 0 & 4 \\
5 &    8 & (bccbac)  & 1 & 0 & 0 & 0 & 1 & 4 & 12 & 20 & 10 & 4 &  6 & 4 & 2
\end{array} $$
\end{table}

\medskip

As explained in~\cite{Ho-kleinian}, Section~7, a marking of a Kleinian code is the analog of
the concept of a marking of a binary code, a $D_1$-frame of a lattice or a Virasoro frame
of a vertex operator algebra. Continuing this analogy with $L$-codes, one sees that for
an $L$-code there would be only one choice of what one could call a marking of an $L$-code.
Thus there is no additional structure here to consider for $L$-codes.


\subsection{The map $\sigma$}\label{sigmasection}
In this section, we consider a map between $L$-codes and Kleinian codes which doubles
the length.

\begin{defi}
For an $L$-code $C$ of length $n$, the Kleinian code $\sigma(C)$
of length~$2n$ is defined by $\sigma(C):=\widetilde{C}+\delta_2^n,$
where \ $\widetilde{\hskip0in}:L^n\rightarrow K^{2n}$ is the map defined
coordinate-wise:
$$(c_1,\,\ldots,\,c_n)\mapsto (d_{1,1},\,d_{1,2},\,\ldots,\,d_{n,1},\,d_{n,2})$$
with $c_i\mapsto (d_{i,1},\,d_{i,2})$ is determined by
$0\mapsto(0,0)$, $1\mapsto(0,a)$, $\omega\mapsto(b,b)$
and $\bar{\omega}\mapsto(b,c)$,
and where $\delta_2^n=\{(0,0),(a,a)\}^n$.
\end{defi}
Thus, every codeword in $C$ is replaced with $2^n$ codewords in $K^{2n}$.

\begin{rem}
Two $L$-codes $C$ and $D$ are equivalent if and only if the Kleinian codes
$\sigma(C)$ and $\sigma(D)$ are equivalent.
\end{rem}

\begin{rem}
$\sigma(C\oplus D)=\sigma(C)\oplus \sigma(D)$.
\end{rem}

\begin{rem}
The code $\sigma(C)$ is a linear when $C$ is a linear code.
\end{rem}
Indeed, let $x$, $y\in L^n$. For each
$s\in\sigma(x_i)$ and $t\in\sigma(y_i)$, there exists an
\hbox{$u\in\sigma(x_i+y_i)$} such that $s+t=u$.

\begin{rem}\label{sigmaident}
One has $\sigma(L^n)= \left(\delta_2^{\bot}\right)^n$
and $\sigma(\{{\bf 0}\})=\delta_2^n,$ thus the map $\sigma$ provides
an identification of $L^n$ with $\left(\delta_2^{\bot}/\delta_2\right)^n$.
\end{rem}
\begin{proof} Note that $\sigma(L)=\{(0,0),(0,a),(a,0),(a,a),(b,b),(b,c),(c,b),(c,c)\}=\delta_2^{\bot}.$
\end{proof}

\begin{lem}\label{Wsigma} We have $${\rm W}_{\sigma(C)}(u,v)={\rm swe}_C(u^2+v^2,2uv,2v^2).$$
\end{lem}
\begin{proof} The map $\sigma$ sends $\{0\}\subset L$ to $\{(0,0),\, (a,a)\}$,
$\{1\}$ to $\{(0,a),\, (a,0)\}$, $\{\omega\}$ to $\{(b,b),\,(c,c)\}$, and
$\{\bar{\omega}\}$ to $\{(b,c),\,(c,b)\}$.
%
Thus, a codeword represented by $x^ry^sz^t$ in the
symmetric weight enumerator of $C$ is mapped to $2^n$ codewords,
represented by $(u^2+v^2)^r(2uv)^s(2v^2)^t$ in the Hamming weight
enumerator of $\sigma(C)$.
\end{proof}

\begin{lem}\label{simgaprop} The map $\sigma$ maps self-dual $L$-codes
to self-dual Kleinian codes, and even codes to even codes.
\end{lem}
\begin{proof} First note that for $x$, $y\in L^n$ one has
$x_i\cdot y_i = z_i\cdot w_i$, for all $z_i\in \sigma(x_i)$ and
$w_i \in \sigma(y_i)$.
This implies that
$$ (x, y) = (z, w) \mbox{ for all }z\in \sigma(x)\mbox{ and }w\in \sigma(y).$$
Thus the identification of $L^n$ with $\left(\delta_2^{\bot}/\delta_2\right)^n$
in~Remark~\ref{sigmaident} identifies the scalar product on $L^n$ with the induced
scalar product on  $\left(\delta_2^{\bot}/\delta_2\right)^n$. In particular,
self-orthogonal codes $C$ are mapped to self-orthogonal codes $\sigma(C)$.
If $C$ is self-dual of length~$n$, it has dimension $\frac{n}{2}$, i.e.~contains $2^n$ codewords
(see Lemma~\ref{dimensionSelfDual2}). Then, $\sigma(C)$ of length $2n$ contains $2^n\cdot2^n=4^n$ codewords,
i.e.~has dimension $n$. Therefore the self-orthogonal code $\sigma(C)$ is self-dual.

For the last statement, suppose $C$ is an (Euclidean) even $L$-code.
Then each codeword contributes a term of the form $x^ry^{2s}z^t$ in the symmetric weight enumerator.
By Lemma~\ref{Wsigma}, this word is mapped to $2^n$ codewords, represented by
$(u^2+v^2)^r(2uv)^{2s}(2v^2)^t$ in the Hamming weight enumerator.
The exponent of $v$ for each codeword is even;
thus, the Kleinian code $\sigma(C)$ is (Hamming) even.
\end{proof}

\begin{flushleft}
\textbf{Examples:}

$\sigma(\Xi_1)=\sigma(\{(0),\,(\omega)\})=\{(0,0),\,(a,a),\,(b,b),\,(c,c)\}=\epsilon_2$.

$\sigma(\Upsilon_2)=\sigma(\{(0,0),\,(\omega,1),\,(1,\omega),\,(\bar{\omega},\,\bar{\omega})\})=(\delta_2^2)^+$

\end{flushleft}


\subsection{The map $\psi$}
In this section, we consider a map between $L$-codes which doubles the length.

\begin{defi}
For an $L$-codes $C$ of length $n$, the  $L$-code $\psi (C)$
of length~$2n$ is defined by $\psi(C)=\overline{C}+\Delta_2^n,$
where \ $\overline{\vphantom{C}\hskip.1in}:L^n\rightarrow L^{2n}$ is the map defined
coordinate-wise:
$$(c_1,\,\ldots,\,c_n)\mapsto (d_{1,1},\,d_{1,2},\,\ldots,\,d_{n,1},\,d_{n,2})$$
with $c_i\mapsto (d_{i,1},\,d_{i,2})$ is determined by
$0\mapsto(0,0)$, $1\mapsto(0,1)$, $\omega\mapsto(\o,\o)$
and $\bo\mapsto(\o,\bo)$,
and where $\Delta_2^n=\{(0,0),(1,1)\}^n$.
\end{defi}
Thus, every codeword in $C$ is replaced with $2^n$ codewords in $L^{2n}$.

Essentially, $\psi$ is defined such that $\phi\circ\psi=\sigma$.
This determines $\psi$ uniquely since $\phi^{-1}$ exists.

\begin{rem}\label{psiident}
One has $\psi(L^n)= \left(\Delta_2^{\bot}\right)^n$
and $\psi(\{{\bf 0}\})=\Delta_2^n,$ thus the map $\psi$ provides
an identification of $L^n$ with $\left(\Delta_2^{\bot}/\Delta_2\right)^n$.
\end{rem}

\begin{cor} The map $\psi$ maps self-dual codes to self-dual codes,
even codes to even codes, and
$${\rm W}_{\psi(C)}(u,v)={\rm swe}_{C}(u^2+v^2,2uv,2v^2),$$
$${\rm swe}_{\psi(C)}(x,y,z)={\rm swe}_ C(x^2+y^2,2xy,2z^2).$$
\end{cor}
\begin{proof} The properties follow directly from the properties of
$\phi$ and $\sigma$. The proofs of the
weight enumerator identities are clear. The map $\psi$ maps
self-dual codes to self-dual codes because both $\sigma$ and
$\phi^{-1}$ do. To show that $\psi$ maps even codes to even codes,
consider a codeword in an even code $C$ that is represented by
$x^ry^{2s}z^t$ in the symmetric weight enumerator of $C$. This
codeword is mapped to $2^n$ codewords, represented by
$(x^2+y^2)^r(2xy)^{2s}(2z^2)^t$ in the symmetric weight enumerator
of $\psi(C)$. The exponent of $y$ in each term of
$(x^2+y^2)^r(2xy)^{2s}(2z^2)^t$ is divisible by $2$.
Thus, $\psi(C)$ must be even.
\end{proof}

\subsection{The map $\beta$}\label{binarysection}

In this section, we consider a map between $L$-codes and binary codes
which triples the length.

\smallskip

The vector space ${\mathbf F}_2^n$ comes with a canonical scalar product.
Each $L$-code can be viewed naturally as a binary code,
because $L^n \cong {\mathbf{F}}_2^{2n}$ as additive
groups.
However, the concepts of weight and scalar product are not
the same, so this map is not very interesting.
We define the map $\beta$ as in~\cite{CRSS-quant}.

\begin{defi}
The map $\beta:L^n\rightarrow {\mathbf F}_2^{3n}$ is defined coordinate-wise:
$$(c_1,\,c_2,\,\ldots,\,c_n)\mapsto (c_{1,1},\,c_{1,2},\,c_{1,3},\,c_{2,1},\,c_{2,3},\,c_{3,3},\,\ldots,\,c_{n,1},\,c_{n,2},\,c_{n,3})$$
where $c_i\mapsto (c_{i,1},\,c_{i,2},\,c_{i,3})$ is determined by $0 \mapsto (0,0,0)$,
$1 \mapsto (0,1,1)$, $\omega \mapsto (1,0,1)$ and  $\bar{\omega} \mapsto (1,1,0)$.
\end{defi}

\begin{rem}\label{propertiesofbeta} The map $\beta$ is a group homomorphism.
In particular, the map $\beta$ sends a linear $[n,k]$ $L$-code $C$ to a
$[3n,2k]$ linear binary code $\beta(C)$.
\end{rem}

\begin{lem}\label{scalarproductbeta} The map $\beta$ preserves scalar products.
In particular, if $C$ is a self-orthogonal code, then $\beta(C)$ is
a self-orthogonal code.
\end{lem}
\begin{proof}
For any $c_i$, $d_i\in L$, one checks that
$$c_i\cdot d_i=\sum_{j=1}^3 c_{i,j} \cdot d_{i,j},$$ where the
product in ${\mathbf F}_2$ is the field multiplication.
Now it follows from the definition of the scalar product
that for $c$, $d\in L^n$ one has $(c,d)=\left(\beta(c),\beta(d)\right)$.

The last statement follows.
\end{proof}

\begin{lem}[\cite{CRSS-quant}]\label{betaAutomorphism}
For the automorphism group of an $L$-code $C$ one has
$${\rm Aut}(C)\cong{\rm Aut}(\beta(C))\cap {\rm Aut} (\beta (L^n)).$$
\end{lem}
Lemma~\ref{betaAutomorphism} provides an easy way
to determine the automorphism group of $L$-codes of small lengths with
Magma (\cite{magma}), since Magma is able to compute the
automorphism groups of binary codes.


\section{Classification of self-dual codes}\label{selfdual}


\subsection{Weight enumerators of self-dual codes}

The weight enumerators of self-dual codes belong to a certain subring of invariant polynomials
which we will determine in this section.

We first describe the relation between the weight enumerator of a code and its dual.
Since the map $\phi$ between $L^n$ and $K^n$ is an isomorphism and the scalar products
for $L^n$ and $K^n$ are the same, the next two theorems follow from the corresponding
results for Kleinian codes, cf. Remark~\ref{Wphi} and~\cite{Ho-kleinian}, Theorem~1 and~2.

\begin{thm}[Mac-Williams identity for Hamming weight enumerators]\label{macwilliamsW}
$${\rm W}_{C^{\bot}}(u,v)=\frac{1}{|C|}{\rm W}_C(u+3v,u-v).$$
\end{thm}

\begin{thm}[Mac-Williams identity for symmetrized weight enumerators]\label{macwilliamsswe}
$${\rm swe}_C(x,y,z)=\frac{1}{|C|}{\rm swe}_{C^{\bot}}(x+y+2z,x+y-2z,x-y).$$
\end{thm}

\begin{proof} Apply the identity ${\rm swe}_C(x,y,z)={\rm cwe}_C(x,y,z,z)$
to Theorem~2 in~\cite{Ho-kleinian}, the Mac-Williams identity for complete weight enumerators.
\end{proof}

There is no such identity for Euclidean weight enumerators.

\begin{thm} \label{ThmEvenSwe3poly}
Let $C$ be an even self-dual code of length~$n$. Then the symmetrized
weight enumerator ${\rm swe}_C$ is a weighted homogeneous polynomial
of degree $n$ in the symmetric weight enumerators
${\rm swe}_{\Xi_1}=x+z$, ${\rm swe}_{\Delta_2^+}=x^2+y^2+2z^2$, and
${\rm swe}_{\Upsilon_3}=x^3+3xz^2+3y^2z+z^3$.
\end{thm}

\begin{proof} 
The polynomial ${\rm swe}_C(x,y,z)$ is invariant
under the substitutions given by the two matrices
$$S=\left( \begin{array}{rrr}
     \textstyle  \frac{1}{2} & \textstyle\frac{1}{2} & 1 \\
     \textstyle \frac{1}{2} & \textstyle \frac{1}{2} & -1 \\
     \textstyle  \frac{1}{2} &\textstyle -\frac{1}{2} & 0 \\
     \end{array} \right)
\qquad\hbox{and}\qquad
T= \left(\begin{array}{rrr}
    1 & 0 & 0 \\
    0 & -1 & 0 \\
    0 & 0 & 1 \\
  \end{array} \right)_.$$
For $S$ this follows from Theorem~\ref{ThmEvenSwe3poly} and for $T$ note
that  each codeword must contain an even number of ones.

These two matrices generate a group $\mathcal{G}$ of
order $6$. Let $a_d$ be the number of
linearly independent, homogeneous invariant polynomials of degree $d$. By the
theorem of Molien (\cite{Molien}, cf.~also~\cite{MacSl}) one has for
$\Phi(\lambda):=\sum_{i=0}^{\infty}a_i\lambda^i$ the identity
$$\Phi(\lambda)=\frac{1}{|\mathcal{G}|}\sum_{A\in\mathcal{G}}\frac{1}{\mbox{det}(I-\lambda A)}.$$
A calculation shows that the group $\mathcal{G}=\langle S,\,T\rangle$ has Molien series
$$\Phi(\lambda)=\frac{1}{(1-\lambda)(1-\lambda^2)(1-\lambda^3)}.$$

The polynomials ${\rm swe}_{\Xi_1}=x+z$,
${\rm swe}_{\Delta_2^+}=x^2+y^2+2z^2$, and
${\rm swe}_{\Upsilon_3}=x^3+3xz^2+3y^2z+z^3$
are algebraically independent because the Jacobian determinant
$$J=\left|\frac{\partial({\rm swe}_{\Xi_1},{\rm swe}_{\Delta_2^+},{\rm swe}_{\Upsilon_3})}{\partial(x,y,z)}\right|=24xyz-6x^2y-24yz^2$$
is not identically zero (\cite{Yu-relations}).
The degrees of these polynomials correspond to the Molien series, so
these polynomials freely generate the ring of all invariants.
\end{proof}

This can be used to describe the Euclidean weight enumerators.
\begin{cor}
Let $C$ be an even self-dual code of length~$n$. Then the Euclidean weight
enumerator ${\rm EW}_C(a,b)$ is a weighted homogeneous polynomial of
weight $2n$ in $a^2+b^2$, $a^4+a^2b^2+2b^4$, and $a^6+6a^2b^4+b^6$,
or equivalently in the weight enumerators of $\Xi_1$, $\Delta_2^+$,
and $\Upsilon_3$. They satisfy the relation
\begin{eqnarray*}\label{ewident}
 \nonumber 2\,{\rm EW}_{\Upsilon_3}^2 & = & 11\,{\rm EW}_{\Xi_1}^6-36\,{\rm EW}_{\Xi_1}^4\,{\rm EW}_{\Delta_2^+}+5\,{\rm EW}_{\Xi_1}^3{\rm EW}_{\Upsilon_3} \\
   &   &\  {}+36\,{\rm EW}_{\Xi_1}^2{\rm EW}_{\Delta_2^+}^2-9\,{\rm EW}_{\Xi_1}{\rm EW}_{\Delta_2^+}{\rm EW}_{\Upsilon_3}-9\,{\rm EW}_{\Delta_2^+}^3.
\end{eqnarray*}
\end{cor}
\begin{proof} 
The first statement follows directly from Theorem~\ref{ThmEvenSwe3poly} and Remark~\ref{enumRel}.
The stated relation is also directly checked.
\end{proof}

Similarly, for self-dual codes alone one has:
\begin{thm} \label{ThmSwe3poly}
Let $C$ be a self-dual code of length~$n$. Then the symmetrized
weight enumerator ${\rm swe}_C$ is a weighted homogeneous polynomial
of degree $n$ in the symmetric weight enumerators
${\rm swe}_{\Gamma_1}=x+y$, ${\rm swe}_{\Xi_1}=x+z$, and ${\rm swe}_{\Delta_2^+}=x^2+y^2+2z^2$.
\end{thm}
\begin{cor}
Let $C$ be a self-dual code of length~$n$. Then the Euclidean weight
enumerator ${\rm EW}_C(a,b)$ is a weighted homogeneous polynomial of
weight $2n$ in ${\rm EW}_{\Gamma_1}=a^2+ab$, ${\rm EW}_{\Xi_1}=a^2+b^2$ and ${\rm EW}_{\Delta_2^+}=a^4+a^2b^2+2b^4$.
\end{cor}

Since the scalar products for $L$-codes and Kleinian codes are the same, it is clear that
for the Hamming weight enumerators of self-dual codes the same result as
for Kleinian codes hold (\cite{Ho-kleinian}, Thm.~3):
The Hamming weight enumerator is a polynomial in  ${\rm W}_{\Gamma_1}=u+v$ and ${\rm W}_{\Delta_2^+}=u^2+3v^2$.
Theorem~\ref{ThmEvenSwe3poly} implies that for even codes no further restrictions arise.


\subsection{Self-dual codes generated by short vectors}

The {\it Euclidean weight-$k$-subcode\/} of a code $C$ is the subcode
generated by all words in $C$ of Euclidean weight less than or equal to $k$.

\begin{thm}\label{thisblows}
\begin{itemize}
\item [(1)]Euclidean weight-$1$-subcodes of a self-orthogonal code $C$
can be split off: $C\cong D \oplus \Gamma_1^l$,
 where $l$ is a non-negative integer, and the minimal
Euclidean weight of $D$ is strictly larger than~$1$.
\item [(2)]The Euclidean weight-$2$-subcode of a self-orthogonal
code $C$ of minimal weight $2$ is equivalent to direct sums of $\Delta_l$, $l\geq 1$,
and $\Xi_1$.
\end{itemize}
\end{thm}
\begin{proof} To prove statement~(1), let $C$ be a self-orthogonal
code with at least one weight-$1$-codeword.
A weight-$1$-codeword in $C$ is equivalent to $(0,\ldots,0,1)$, and
the weight-$1$-subcode it generates is equivalent to $\Gamma_1$. Now
define $C_x:=\{(*,\ldots,*,x)\in C\}$, where $'*'$ could be any
element of $L$. Then $C=C_0\cup C_1 \cup C_{\omega} \cup C_{\bar{\omega}}$.
Since~$C$ is self-orthogonal, the scalar product of $(0,\ldots,0,1)$
and every vector in $C$ must be zero. Thus $C=C_0\cup C_1$ because
$((0,\ldots,0,1),\,(*,\ldots,*,\omega))=((0,\ldots,0,1),\,(*,\ldots,*,\bar{\omega}))=1.$
In fact, $C_1=C_0 + (0,\ldots,0,1)$, so we can write $C=C_0'\oplus \Gamma_1$, where
$C_0'$ is the subcode of $C_0$ with the zeros in the rightmost position
deleted. Repeating the process, 
we find a code $D$ with minimal Euclidean weight larger than~$1$
and $C\cong D \oplus \Gamma_1^l$.

\smallskip

Let $B$ be the weight-$2$-subcode of $C$.
To prove statement~(2), we proceed by induction on the dimension of $B$.
If the dimension of $B$ is zero we are done.
Suppose that $B$ contains at least one non-zero weight-$2$-codeword $c$.
There are two cases to consider:

\smallskip

Case (i). The codeword is equivalent to $(0,\,\ldots,\, 0,\,\omega)$.

Then the codeword generates a copy of $\Xi_1$.
Since $\Xi_1$ is self-dual, this word generates an entire
component by the same argument as in the proof of statement~(1).
The induction step follows.

\smallskip

Case (ii). The codeword is equivalent to $(0,\,\ldots,\, 0,\,1,\,1)$.

Let $B=B'\cup (B'+c)$ for a complement $B'$ of $c$.
Then $B'\cong \Delta_1^{j_1}\oplus\cdots\oplus\Delta_r^{j_r}\oplus\Xi_1^{k}$,
for some integers $j_i\geq 0$, $r\geq 1$ and $k\geq 0$.
The two nonzero entries of $c$ must belong to the support of two different components
$\Delta_{i_1}$ and $\Delta_{i_2}$ of $B'$. Together with $c$ restricted to these two components
they generate a component $\Delta_{i_1+i_2}$. This finishes the induction step.
\end{proof}

\smallskip

The following concepts are discussed in~\cite{CoSl}.

Given a self-orthogonal code $C$, there is a standard way to
generate a larger self-orthogonal code $D$ such that $C\subset D$.
We have $C\subset C^{\bot}$. By choosing any isotropic subgroup of cosets $C^{\bot}/C$
with respect to the quadratic form induced from~$q$,
we obtain an extension $D$ of $C$, with $C\subset D\subset D^{\bot}\subset C^{\bot }$.
All self-dual codes containing $C$ are obtained in this way.

Let $\bar{C}$ be the Euclidean weight-$2$-subcode of $C$.
We call $\Lambda:=C/{\bar C}$ the {\it gluecode\/} of $C$ and note that
$\Lambda\subset {\bar{C}}^{\bot}/{\bar{C}}$. Then
${\rm Aut}(C)=G_0.G_1.G_2$, where $G_0$ are the automorphisms
preserving the components of $\bar C$ and the cosets $\Lambda/\bar C$,
$G_1$ is the quotient of the automorphisms that preserve the components of $\bar C$ by $G_0$, and
$G_2$ is the induced permutation group on the components of ${\bar C}$.

\paragraph{Example $\Delta_k^+$:}\label{deltakplus}
There is an extension of $\Delta_k$ (cf.~Section \ref{exOfLcodes}) to
a code $\Delta_k^+$ so that $\Delta_k^+$ is an even self-dual code.
The group $\Delta_k^{\bot}/\Delta_k$ has $4$ elements,
denoted by $[\,{\bf 0}\,]=(0,\,\ldots,\,0),$
$[\,{\bf a}\,]=(\omega,\,\ldots,\,\omega)$,
$[\,{\bf \bar{a}}\,]=(\bar{\omega},\,\omega,\,\ldots,\,\omega)$
and $[\,{\bf 1}\,]=(1,\,0,\,\ldots,\,0)$.
Then one defines
$\Delta_k^+=\Delta_k\cup(\Delta_k + [\,{\bf a}\,])$,
which is a group of order $2|\Delta_k|$. Hence it is self-dual, and every
element of $\Delta_k + [\,{\bf a}\,]$ has coordinates $\omega$ or
$\bar{\omega}$, so $\Delta_k^+$ is even. To compute
${\rm W}_{\Delta_k^+}(u,v)$, note that each vector in
$\Delta_k + [\,{\bf a}\,]$ is nonzero in each component; thus
$${\rm W}_{\Delta_k^+}(u,v)={\rm
W}_{\Delta_k}(u,v)+2^{k-1}v^k=\frac{1}{2}((u+v)^k+(u-v)^k+2^kv^k).$$
More specifically, each coordinate of a vector in $\Delta_k + [\,{\bf a}\,]$ is
$\omega$ or $\bar{\omega}$. Thus
$$ {\rm swe}_{\Delta_k^+}(x,y,z)={\rm
swe}_{\Delta_k}(x,y,z)+2^{k-1}z^k=\frac{1}{2}((x+y)^k+(x-y)^k+2^kz^k) $$
and
$${\rm EW}_{\Delta_k^+}(a,b)={\rm
EW}_{\Delta_k}(a,b)+2^{k-1}b^{2k}=\frac{1}{2}((a^2+ab)^k+(a^2-ab)^k+2^kb^{2k}). $$
The automorphism group of $\Delta_k^+$ is the same as of $\Delta_k$,
except we can only permute the components $\omega$ and
$\bar{\omega}$ an even number of times. Thus,
${\rm Aut}(\Delta_k^+)=S_2^{k-1}{:}S_k$.

\medskip


\begin{thm}[Relation between even and non-even self-dual codes]\label{evenodd}
\begin{itemize}
\item[(1)]
For $k\geq 1$, there is a one-to-one correspondence between isomorphism classes of pairs
$(C,\Delta_k)$, where $C$ is an even self-dual code of length~$n$ and
$\Delta_k$ is a subcode of $C$ and equivalence classes of self-dual
codes $D$ of length $n-k$.
\item[(2)] The map which assigns to an even self-dual code $C$ of length~$n$ and
a nonzero coset $[x]$ in $L^n/C$ the code $C_0\cup (C_0+y)$ where
$C_0=C\cap x^\bot$ and $y$ is a vector in~$[x]$ of odd Euclidean weight
induces a correspondence between $G$-isomorphism classes of pairs
$(C,[x])$ and equivalence classes of non-even self-dual codes $D$ of length~$n$
such that for a given equivalence class $[D]$ there are either one or two
pairs $[(C,[x])]$.
\end{itemize}
\end{thm}
The proof is similar as for Kleinian codes, cf.~\cite{Ho-kleinian}, p.~235.

We illustrate case (2) for $n=2$:
Representatives for the pairs $(C,[x])$ are $(\Xi_1^2,[(1,0)])$, $(\Xi_1^2,[(1,1)])$,
$(\Delta_2^+,[(1,0)])$ and $(\Delta_2^+,[(\omega,0)])$ corresponding to the non-even codes
$\Gamma_1\Xi_1$, $\Upsilon_2$, $\Gamma_1^2$ and $\Upsilon_2$, respectively.
This is in agreement with Table~\ref{length1-4} below.


\subsection{Mass formulas}

We denote by $M(n)$ be the number of distinct, but possibly equivalent,
self-dual codes $L$-codes of length $n$. Similarly, let $M_e(n)$ be
the number of even self-dual $L$-codes of length $n$.

\begin{thm}\label{massFormula}
We have
$$M(n)=\sum_{[C]}\frac{2^n \cdot n!}{
|{\rm Aut}(C)|}=\prod_{i=1}^n(2^i + 1),$$ where the sum is over
equivalence classes of self-dual codes of length $n$.
\end{thm}

\begin{proof} Consider the action of $G=S_2^n{:}S_n$ on
the set of self-dual $L$-codes of length $n$. Then the right hand
expression for $M(n)$ is derived from the orbit formula.

By Theorem~6 in \cite{Ho-kleinian}, the total number of self-dual Kleinian codes
of length~$n$ is $\prod_{i=1}^n(2^i +1)$.
As discussed in Section \ref{sectionphi}, $\phi$ is a bijective map between
self-dual $L$-codes of length $n$ and self-dual Kleinian codes of
length $n$. This proves that the number of $L$-codes in of length~$n$ must also be
$\prod_{i=1}^n(2^i + 1)$.
\end{proof}


\begin{thm}\label{massFormula2} We have
$$M_e(n)=\sum_{[C]}\frac{2^n \cdot n!
}{|{\rm Aut}(C)|}=\prod_{i=0}^{n-1}(2^i + 1),$$ where the sum is
over equivalence classes of even self-dual codes of length $n$.
\end{thm}
\begin{proof} The first equation is again clear. For the second,
we use that even self-dual $L$-codes of length $n$ can be identified with doubly-even
self-dual binary codes of length~$8n$ containing the $3n$-dimensional
doubly-even self-orthogonal code $d_8^n$ as a subcode. Here, $d_8$ is the binary
code generated by the codewords
$$(1,\,1,\,1,\,1,\,0,\,0,\,0,\,0),\  (0,\,0,\,1,\,1,\,0,\,0,\,0,\,0)\  \hbox{and } (0,\,0,\,0,\,0,\,1,\,1,\,1,\,1).$$
The number of doubly-even self-dual binary codes of length $m$ which
contain a fixed $k$-dimensional doubly-even self-orthogonal subcode
including the overall one vector $(1,\,\ldots,\, 1)$ is
$\prod_{i=0}^{\frac{1}{2}m-k-1}(2^i+1)$ (see~\cite{MacSl}, Ch.~19, Thm.~22).
Thus
$$M_e(n)=\prod_{i=0}^{\frac{1}{2}8n-3n-1}(2^i+1)=\prod_{i=0}^{n-1}(2^i+1).$$
\end{proof}



\subsection{Classification of self-dual codes up to length $10$}\label{SectionClassCodes}

In this section, we classify self-dual $L$ codes up to length~$10$.
We list them explicitly for length up to $4$ and give for larger
lengths the number of such codes according to their minimal Euclidean weight.

\begin{thm}\label{classCodes4} The inequivalent self-dual $L$-codes up to length~$4$
are listed in Table~\ref{length1-4} and~\ref{oddlength4}. The symmetric and
Euclidean weight enumerators of the codes up to length~$3$ and of the even
codes of length~$4$ are given in Table~\ref{poly}.
\end{thm}
The columns of Table \ref{length1-4} and \ref{oddlength4} contain the following information:\newline
\begin{tabular}{rl}
$n$: & The length of the code $C$. \\
No.: & The number of a code of given length. \\
$C$: & The name of $C$ explained below.  \\
type: &  The type: even codes are marked Y, and noneven codes are marked N.  \\
${\rm Aut}(C)$: & The automorphism group of $C$.  \\
$|\,[C]\,|$ : &  The number $\frac{2^n \cdot n!}{|{\rm Aut}(C)|}$ of $L$-codes equivalent to $C$. \\
$\phi(C)$: & The Kleinian code of length $n$ obtained by $\phi$.  \\
$\mathcal{M}$: &  A marking of $\phi(C)$ such that $\phi_\mathcal{M}^{-1}(\phi(C))=C$.  \\
$\sigma(C)$: &  The Kleinian code of length $2n$ obtained by $\sigma$. \\
$\psi(C)$: &  The $L$-code of length $2n$ obtained by $\psi$. \\
\end{tabular}


We use the following naming conventions of $L$-codes.
One symbol is used for an indecomposable subcode of $C$, so that a decomposition is
given in the naming of $C$. The subscripts 
denote the length of a code, and the superscripts denote the number of copies
in the direct sum. A superscript $+$ indicates the gluing of subcodes to
an indecomposable component.
Often, the symbol of an indecomposable subcode $D$ is a capital Greek
letter chosen such that the Kleinian code $\sigma(D)$ is the corresponding
lower case Greek capital letter.

For the names of specific $L$-codes, we refer first to Sections \ref{exOfLcodes}
and~\ref{deltakplus}. Otherwise, the following names are used:

The $L$-code $\Sigma_n$ for $n\geq 3$ is defined to be
$\phi_{\mathcal{M}}^{-1}(\delta_n^+)$ for
$\mathcal{M}=(c,\,c,\,\ldots,\,c)$.
The $L$-code $\mathcal{J}_n$ for $n\geq 3$ is defined to be
$\phi_{\mathcal{M}}^{-1}(\delta_n^+)$ for 
$\mathcal{M}=(b,\,b,\,\ldots,\,b)$ if $n$ is odd and
$\mathcal{M}=(b,\,b,\,\ldots,b,\,\,c)$ if $n$ is even.
Note that $\Sigma_3$ equals $\Upsilon_3$ and $\mathcal{J}_2$ would be $\Upsilon_2$.



The $L$-codes  $\mathcal{P}_3$ and $\mathcal{Q}_3$
are named based on the markings $(a,\,a,\,b)$  and $(a,\,b,\,b)$ 
for $\delta_3^+$, respectively.


We let $\mathcal{D}_{2n}=\phi^{-1}(\sigma(\Sigma_n))=\psi(\Sigma_n)$.

%

The $L$-code $-_n$ stands for the zero-dimensional code of length $n$ and $L$-codes marked "nc" are not named.

For the names of Kleinian codes, we refer to~\cite{Ho-kleinian}.
\medskip


\begin{table}\caption{Self-dual codes of length up to $3$ and even self-dual codes of length~$4$}\label{length1-4}
\center{
\begin{tabular}{|cr|c|l|c|r|cc|c|c|}
  \hline
 $n$& No.&type & $ C$ & ${\rm Aut}(C)$ & $|\,[C]\,|$ &$\phi(C)$ &$\mathcal{M}$ &$\sigma(C)$ & $\psi(C)$ \\ 
  \hline \hline

1 &  1&Y & $\Xi_1$ & $e$ & 2  & $\gamma_1$ & $(b)$ & $\epsilon_2$ & $\Delta_2^+$ \\ %

  &  2&N & $\Gamma_1$ & $2$ & 1  & $\gamma_1$ & $(a)$ &$\gamma_1^2$  & $\Gamma_1^2$ \\ %

\hline

2 &  1 & Y & $\Xi_1^2$ & $S_2$ & 4  & $\gamma_1^2$ & $(bb)$  & $\epsilon_2^2$ & $(\Delta_2^+)^2$ \\

  & 2 & Y & $\Delta_2^+$ & $2{:}S_2$ & 2  & $\epsilon_2$ & $(aa)$ & $\delta_4^+$ & $\Delta_4^+$\\

  & 3 & N &$\Gamma_1^2$ & $2^2{:}S_2$ & 1  & $\gamma_1^2$ & $(aa)$ & $\gamma_1^4$ & $\Gamma_1^4$\\

  & 4 & N &$\Gamma_1\Xi_1$ & ${2}$ & 4 & $\gamma_1^2$& $(ab)$  & $\gamma_1^2\epsilon_2$ & $\Gamma_1^2\Delta_2^+$\\

  & 5 & N & $-_2^+=\Upsilon_2$ & $S_2$ & 4  & $\epsilon_2$& $(ab)$  & $(\delta_2^2)^+$ & $\mathcal{D}_4$\\

\hline

3 & 1 & Y & $\Delta_3^+$ & $2^2{:}S_3$ & 2 & $\delta_3^+$& $(aaa)$ &$\delta_6^+$ & $\Delta_6^+$\\

  & 2 &Y &  $-_3^+=\Upsilon_3=\Sigma_3$ & $S_3$ & 8 & $\delta_3^+$& $(ccc)$ & $(\delta_2^3)^+$ & $\mathcal{D}_6$ \\

  & 3 &Y & $\Delta_2^+\Xi_1$ & ${2}{:}S_2$& 12 &$\epsilon_2\gamma_1$& $(aab)$ & $\delta_4^+\epsilon_2$ & $\Delta_4^+\Delta_2^+$ \\

  & 4 &Y & $\Xi_1^3$ & $S_3$ & 8 & $\gamma_1^3$ & $(bbb)$& $\epsilon_2^3$ & $(\Delta_2^+)^3 $\\

  & 5 & N &  $-_3^+=\mathcal{J}_3$ & $S_3$ & 8 & $\delta_3^+$ & $(bbb)$ & $(\delta_2^3)^+$& nc \\

  & 6 & N & $-_3^+=\mathcal{Q}_3$ & $S_2$ & 24 & $\delta_3^+$ & $(abb)$ & $(\delta_2^2\delta_2)^+$ & nc \\

  & 7 & N & $(\Delta_2\,-_1)^+ =\mathcal{P}_3$ & $2{:}S_2$ & 12 & $\delta_3^+$ & $(aab)$ & $(\delta_4\delta_2)^+$ & nc \\

  & 8 & N  & $\Delta_2^+\Gamma_1$ & $2{:}S_2\times 2$ & 6 & $\epsilon_2\gamma_1$ & $(aab)$ & $\delta_4^+\gamma_1^2$ & $\Delta_4^+\Gamma_1^2$\\

  & 9 & N & $\Upsilon_2\Gamma_1$ & $S_2\times2$ & 12 & $\epsilon_2\gamma_1$ & $(aba)$ & $(\delta_2^2)^+\gamma_1^2$ & $\mathcal{D}_4\Gamma_1^2$\\

  & 10 & N & $\Upsilon_2\Xi_1$ & $S_2$ & 24 & $\epsilon_2\gamma_1$ & $(abb)$ & $(\delta_2^2)^+\epsilon_2$ & $\mathcal{D}_4\Delta_2^+$ \\

  & 11 &N  & $\Gamma_1^3$ & $2^3{:}S_3$ & 1 & $\gamma_1^3$ & $(aaa)$ & $\gamma_1^6$ & $\Gamma_1^6$\\

  & 12 &N  & $\Gamma_1^2\Xi_1$ & $2^2{:}S_2$ & 6 & $\gamma_1^3$ & $(aab)$ & $\gamma_1^4\epsilon_2$ & $\Gamma_1^4\Delta_2^+$\\

  & 13 &N  & $\Gamma_1\Xi_1^2$ & $2\times S_2$ & 12 & $\gamma_1^3$ & $(abb)$ & $\gamma_1^2\epsilon_2^2$ & $\Gamma_1^2(\Delta_2^+)^2$\\

\hline

4 & 1  & Y & $\Delta_4^+$ & $2^3{:}S_4$ & 2 & $\delta_4^+$ & $(aaaa)$  & $\delta_8^+$ & $\Delta_8^+$\\

  & 2  & Y& $-_4^+=\Sigma_4$ & $S_4$ & 16 & $\delta_4^+$ & $(cccc)$  & $(\delta_2^4)^+_a$ & $\mathcal{D}_8$ \\

  & 3  & Y& $(\Delta_2\,-_2)^+$ & $(2{:}S_2).S_2$ & 48 & $(\delta_2^2)^+$ & $(aabc)$  & $(\delta_4\delta_2^2)^+$ & nc \\

  & 4  & Y& $-_4^+$ & $S_2^3$ & 48 & $(\delta_2^2)^+$ & $(bbbb)$ & $(\delta_2^4)^+_b$ & nc \\

  & 5  & Y&$(\Delta_2^+)^2$ & $(2{:}S_2)^2{:}S_2$ & 12 & $\epsilon_2^2$  & $(aaaa)$ & $(\delta_4^+)^2$ & $(\Delta_4^+)^2$\\

  & 6  & Y&$\Xi_1^4$ & $S_4$ & 16  & $\gamma_1^4$ & $(bbbb)$ & $\epsilon_2^4$ & $(\Delta_2^+)^4$ \\

  & 7  & Y&$\Delta_2^+\Xi_1^2$ & $(2{:}S_2)\times S_2$ & 48 & $\epsilon_2\gamma_1^2$ & $(aabb)$  & $\delta_4^+\epsilon_2^2$ & $\Delta_4^+(\Delta_2^+)^2$\\

  & 8  & Y& $\Delta_3^+\Xi_1$ & $2^2{:}S_3$ & 16 & $\delta_3^+\gamma_1$  & $(aaab)$ & $\delta_6^+\epsilon_2$ & $\Delta_6^+\Delta_2^+$\\

  & 9  & Y&$\Upsilon_3\Xi_1$ & $S_3$ & 64 & $\delta_3^+\gamma_1$ & $(cccb)$  & $(\delta_2^3)^+\epsilon_2$ & $\mathcal{D}_6\Delta_2^+$\\
\hline
\end{tabular}}
\end{table}

\begin{table}\caption{The non-even self-dual codes of length $4$}\label{oddlength4}
\center{
\begin{tabular}{|cr|c|l|c|r|cc|}
  \hline
$n$ &  No. & type & $C$ & ${\rm Aut}(C)$ & $|\,[C]\,|$  & $\phi(C)$  & $\mathcal{M}$ \\ 
  \hline
  \hline
4 &  10 & N& $(\Delta_3\,-_1)^+$ & $2^2{:}S_3$ & 16 & $\delta_4^+$ & $(aaab)$ \\ 

  &  11 & N& $(\Delta_2\,-_2)^+$ & $2.S_2^2$ & 48 & $\delta_4^+$ & $(aabb)$ \\ 

  &  12 & N& $-_4^+$ & $S_3$ & 64 & $\delta_4^+$ & $(abbb)$ \\ 

  &  13 & N& $-_4^+=\mathcal{J}_4$ & $S_4$ & 16 & $\delta_4^+$ & $(bbbc)$ \\ 

  &  14 & N&$\Delta_3^+\Gamma_1$ & $(2^2{:}S_3)\times S_2$ & 8 & $\delta_3^+\gamma_1$&$(aaaa)$  \\ 

  &  15 & N&$\mathcal{J}_3\Gamma_1$ & $S_3\times S_2$ & 32 &$\delta_3^+\gamma_1$ &$(bbba)$ \\ 

  &  16 & N&$\mathcal{J}_3\Xi_1$ & $S_3$ & 64 & $\delta_3^+\gamma_1$ &$(bbbb)$ \\ 

  &  17 & N&$\mathcal{Q}_3\Gamma_1$ & $S_2 \times 2$ & 96 & $\delta_3^+\gamma_1$ &$(abba)$ \\ 

  &  18 & N&$\mathcal{Q}_3\Xi_1$ & $S_2$ & 192 & $\delta_3^+\gamma_1$ &$(abbb)$ \\ 

  &  19 & N&$\mathcal{P}_3\Gamma_1$ & $(2{:}S_2)\times 2$ & 48 & $\delta_3^+\gamma_1$ &$(aaba)$ \\ 

  &  20 & N&$\mathcal{P}_3\Xi_1$ & $2{:}S_2$ & 96 & $\delta_3^+\gamma_1$ &$(aabb)$ \\ 

  &  21 & N& $\Upsilon_3\Gamma_1$ & $S_3\times 2$ & 32 &$\delta_3^+\gamma_1$ &$(ccca)$ \\ 

  &  22 & N& $(\Delta_2^2)^+=\mathcal{D}_4$ & $(2{:}S_2)^2{:}S_2$& 12 & $(\delta_2^2)^+$ & $(aaaa)$ \\ 

  &  23 & N& $(\Delta_2\,-_2)^+$ & $2{:}S_2$& 96 & $(\delta_2^2)^+$ & $(aaab)$\\ 

  &  24 & N& $(\Delta_2\,-_2)^+$ & $(2{:}S_2)\times S_2$& 48 & $(\delta_2^2)^+$ & $(aabb)$ \\ 

  &  25 & N& $-_4^+$  & $S_2$ & 192 & $(\delta_2^2)^+$ & $(abab)$ \\ 

  &  26 & N& $-_4^+$  & $S_2$ & 192 & $(\delta_2^2)^+$ & $(abbb)$ \\ 

  &  27 & N& $-_4^+$  & $S_2$ & 192 & $(\delta_2^2)^+$ & $(abbc)$ \\ 

  &  28 & N& $-_4^+$  & $S_2\times S_2$ & 96 & $(\delta_2^2)^+$ & $(bbbc)$ \\ 

  &  29 & N& $-_4^+$  & $S_2^2.S_2$ & 48 & $(\delta_2^2)^+$ & $(bcbc)$ \\ 

  & 30 & N&$\Delta_2^+\Upsilon_2$ & $(2{:}S_2)\times S_2$ & 48 & $\epsilon_2^2$&$(aaab)$ \\ 

  &  31 & N&$\Upsilon_2^2$ & $S_2^2{:}S_2$ & 48 & $\epsilon_2^2$ &$(abab)$ \\ 

  & 32 & N&$\Delta_2^+\Gamma_1^2$ & $(2{:}S_2)\times (2^2{:}S_2)$ & 12 & $\epsilon_2\gamma_1^2$ &$(aaaa)$ \\ 

  & 33 & N&$\Delta_2^+\Gamma_1\Xi_1$ & $2{:}S_2\times 2$ & 48 & $\epsilon_2\gamma_1^2$ &$(aaab)$\\ 

  & 34 & N&$\Upsilon_2\Gamma_1^2$ & $S_2\times (2^2{:}S_2)$& 24 &$\epsilon_2\gamma_1^2$ &$(abaa)$\\ 

  & 35 & N&$\Upsilon_2\Gamma_1\Xi_1$ & $S_2\times 2$ & 96 & $\epsilon_2\gamma_1^2$ &$(abab)$ \\ 

  & 36 & N&$\Upsilon_2\Xi_1^2$ & $S_2\times S_2$ & 96 & $\epsilon_2\gamma_1^2$ &$(abbb)$\\ 

  & 37 & N&$\Gamma_1^4$ & $2^4{:}S_4$ & 1  & $\gamma_1^4$ &$(aaaa)$\\ 

  & 38 & N&$\Gamma_1^3\Xi_1$ & $2^3{:}S_3$ & 8 & $\gamma_1^4$ &$(aaab)$ \\ 

  & 39 & N&$\Gamma_1^2\Xi_1^2$ & $ (2^2{:}S_2)\times S_2$ & 24 & $\gamma_1^4$&$(aabb)$ \\ 

  &  40 & N&$\Gamma_1\Xi_1^3$ & $2\times S_3$ & 32 & $\gamma_1^4$ &$(abbb)$ \\ 
  \hline
\end{tabular}}
\end{table}

\begin{table}\caption{Weight enumerators of self-dual codes of length up to $3$}\label{poly}\small
\center{
\begin{tabular}{|c|c|l|l|}
  \hline
  No. & $C$ & ${\rm swe}_C(x,y,z)$ & ${\rm EW}_C(a,b)$ \\
  \hline
  \hline

1 &  $\Xi_1$ & $x+z$ & $a^2+ b^2$ \\

2 & $\Gamma_1$ & $x+y$ & $a^2+ab$   \\

\hline

  1 & $\Xi_1^2$ &  $x^2+2xz+z^2$ & $a^4+2a^2b^2+b^4$ \\

  2 & $\Delta_2^+$ & $x^2+y^2+2z^2$ & $a^4+a^2b^2+2b^4$ \\

  3 & $\Gamma_1^2$ & $x^2+2xy+y^2$ & $a^4+2a^3b+a^2b^2$ \\

  4 & $\Gamma_1\Xi_1$ & $x^2+xy+xz+yz$ & $a^4+a^3b+a^2b^2+ab^3$ \\

  5 & $\Upsilon_2$  & $x^2+2yz+z^2$ & $a^4+2ab^3+b^4$ \\

\hline

  1 &$\Delta_3^+$ & $x^3+3xy^2+4z^3$   & $a^6+3a^4b^2+4b^6$ \\

  2 &$\Upsilon_3$ & $x^3+3y^2z+3xz^2+z^3$  & $a^6+6a^2b^4+b^6$\\

  3 &$\Delta_2^+\Xi_1$ & $x^3+xy^2+2xz^2+x^2z+y^2z+2z^3$ & $a^6+2a^4b^2+3a^2b^4+2b^6$\\

  4 &$\Xi_1^3$ & $x^3+3x^2z+3xz^2+z^3$  & $a^6+3a^4b^2+3a^2b^4+b^6$ \\

  5 &$\mathcal{J}_3$ & $x^3+y^3+3xz^2+3yz^2$   & $a^6+a^3b^3+3a^2b^4+3ab^5$ \\

  6 &$\mathcal{Q}_3$ & $x^3+xy^2+2xyz+2yz^2+2z^3$ & $a^6+a^4b^2+2a^3b^3+2ab^5+2b^6$\\

  7 &$\mathcal{P}_3$ & $x^3+2xyz+y^2z+xz^2+2yz^2+z^3$   & $a^6+2a^3b^3+2a^2b^4+2ab^5+b^6$ \\

  8 &$\Delta_2^+\Gamma_1$ & $x^3+x^2y+xy^2+2xz^2+y^3+2yz^2$  &   $a^6+a^5b+a^4b^2+a^3b^3+2a^2b^4+2ab^5$\\

  9 &$\Upsilon_2\Gamma_1$ &  $x^3+x^2y+2xyz+xz^2+2y^2z+yz^2$ & $a^6+a^5b+2a^3b^3+3a^2b^4+ab^5$ \\

  10 &$\Upsilon_2\Xi_1$ & $x^3+x^2z+2xyz+xz^2+2yz^2+z^3$  & $a^6+a^4b^2+2a^3b^3+a^2b^4+2ab^5+b^6$\\

  11 &$\Gamma_1^3$ &$x^3+3x^2y+3xy^2+y^3$  & $a^6+3a^5b+3a^4b^2+a^3b^3$ \\

  12 &$\Gamma_1^2\Xi_1$ &  $x^3+2x^2y+x^2z+xy^2+2xyz+y^2z$ & $a^6+2a^5b+2a^4b^2+2a^3b^3+a^2b^4$\\

  13 &$\Gamma_1\Xi_1^2$ & $x^3+x^2y+2x^2z+2xyz+xz^2+yz^2$  & $a^6+a^5b+2a^4b^2+2a^3b^3+a^2b^4+ab^5$ \\

\hline

  1 &$\Delta_4^+$ & $x^4+6x^2y^2+y^4+8z^4$ & $a^8 +6a^6b^2+a^4b^4+8b^8$ \\

  2 & $\Sigma_4$ &  $x^4+y^4+6x^2z^2+6y^2z^2+2z^4$ & $a^8+7a^4b^4+6a^2b^6+2b^8$ \\

  3 & $(\Delta_2\,-_2)^+$ & $x^4+x^2y^2+4xy^2z+x^2z^2+3y^2z^2+4xz^3+2z^4$ & $a^8+a^6b^2+5a^4b^4+7a^2b^6+2b^8$\\

  4 & $-_4^+$      &  $x^4+4xy^2z+2x^2z^2+4xz^3+4y^2z^2+z^4$     & $a^8+6a^4b^4+8a^2b^6+b^8$\\

  5&$(\Delta_2^+)^2$ & $x^4+2x^2y^2+4x^2z^2+y^4+4y^2z^2+4z^4$    & $a^8+2a^6b^2+5a^4b^4+4a^2b^6+4b^8$ \\

  6 & $\Xi_1^4$ & $x^4+4x^3z+6x^2z^2+4xz^3+z^4$& $a^8+4a^6b^2+6a^4b^4+4a^2b^6+b^8$ \\

  7 &$\Delta_2^+\Xi_1^2$ & $x^4+x^2y^2+2x^3z+3x^2z^2$ &\\
 & & \qquad\qquad\qquad $+2xy^2z+4xz^3+y^2z^2+2z^4$    & $a^8+3a^6b^2+5a^4b^4+5a^2b^6+2b^8$\\

  8 &$\Delta_3^+\Xi_1$ &  $x^4+x^3z+3x^2y^2+3xy^2z+4xz^3+4z^4$    & $a^8+4a^6b^2+3a^4b^4+4a^2b^6+4b^8$ \\

  9 &$\Upsilon_3\Xi_1$ &  $x^4+x^3z+3xy^2z+4xz^3+3x^2z^2+3y^2z^2+z^4$    & $a^8+a^6b^2+6a^4b^4+7a^2b^6+b^8$ \\
  \hline
\end{tabular}}\end{table}

\begin{proof}
The main tool to find all $L$-codes used was to compute the inequivalent markings of the
classified Kleinian codes using Lemma~\ref{one2one}.

The indecomposable self-dual Kleinian codes of
lengths up to~$4$ are the six codes $\gamma_1$, $\epsilon_2$,
$\delta_2^+$, $\delta_3^+$, $\delta_4^+$, and $(\delta_2^2)^+$
(\cite{Ho-kleinian}, Theorem~9 and~10 and Table~1 and~2).
Every inequivalent self-dual Kleinian code up
to length $4$ is a direct sum of one or more of these six indecomposable
codes.
If a Kleinian code $D$ is indecomposable, then $\phi_{\mathcal{M}}^{-1}(D)$
is an indecomposable $L$-code for every marking $\mathcal{M}$. The equivalence classes of markings of a
self-dual Kleinian code $D$ are in one-to-one correspondence with the equivalence classes
of self-dual $L$-codes $C$ with  $\phi(C)=D$ (Lemma~\ref{one2one}),
and a marking determines an $L$-code via the map $\phi_{\mathcal{M}}^{-1}$.

The equivalence classes of markings for each of the above self-dual Kleinian code were calculated
and for each case the automorphism group was computed by hand.
The automorphism group of an $L$-code $C=\phi_\mathcal{M}^{-1}(D)$ is the automorphism group
of the corresponding marking $\mathcal{M}$ (cf.~the discussion of markings in Section~\ref{sectionphi}).
To check that the result is correct, the image of ${\rm Aut}(C)$ in $S_{3n}$
using the map $\beta$ was computed with the help of Magma \cite{magma}
by using Lemma~\ref{betaAutomorphism}.

For small $n$, the Hamming weight enumerators of Kleinian codes are
distinct, so these can be used to identify $\phi(C)$ since
${\rm W}_{\phi(C)}={\rm W}_{C}$ and the Hamming weight enumerators of Kleinian 
codes are given in Theorem~10 of \cite{Ho-kleinian}.

\smallskip

The images of the other maps $\sigma$ and $\psi$ and the
weight enumerators were also computed by hand.
\end{proof}

There are several checks to show that this classification is complete.

One check is illustrated in the following table, which is an example
of how to check that all $L$-codes $C$ with $\phi(C)=D$ have been
found, for some Kleinian code~$D$. The example shows which
$L$-codes arise from markings of the Kleinian code $\delta_3^+$.
This Kleinian code has the automorphism group ${\rm Aut}(\delta_3^+)=S_2^2{:}S_3$ of
order~$24$ (\cite{Ho-kleinian}).
The table below shows that there are exactly five $L$-codes mapped to $\delta_3^+$ under $\phi$,
one for each inequivalent marking.
\begin{center}\renewcommand{\arraystretch}{1.2}
\begin{tabular}{ccr}
  $\mathcal{M}$ & ${\rm Aut}(\mathcal{M})$ & $\frac{|{\rm Aut}(\delta_3^+)|}{|{\rm Aut}(\mathcal{M})|}$ \\ 
  \hline
  $(aaa)$ & $2^2{:}S_3$ & 1   \\
  $(ccc)$ & $S_3$ & 4  \\
  $(bbb)$ & $S_3$ & 4 \\
  $(abb)$ & $S_2$ & 12   \\
  $(aab)$ & $2{:}S_2$ & 6 \\
  \hline
  total & & $27$
\end{tabular}\end{center}
The third column in the table checks the marking mass formula (Lemma~\ref{countMarkings}).
The total number of markings adds indeed up to $3^3=27$.
We could also use Lemma~\ref{Phiorbitformula} instead.

Two other checks are Theorem~\ref{massFormula} and
Theorem~\ref{massFormula2}, the formula for the total number of
self-dual $L$-codes and self-dual even $L$-codes, respectively. The
information for the check is provided in the sixth column of Tables~\ref{length1-4}
and~\ref{oddlength4}.
The sum of the sixth columns for each length match the total number of such codes.\nopagebreak

Another check is provided by Theorem~\ref{evenodd}.
\bigskip



\begin{thm} The number of self-dual $L$-codes and even self-dual $L$-codes of length up to~$10$
and fixed minimal weight are given in Table~\ref{tntable} and Table~\ref{entable}, respectively.
\end{thm}


\begin{table}\caption{Number of inequivalent self-dual codes of length $n$ by minimum weight}\label{tntable}
$$\begin{array}{r|rrrrrrrrrrr}
d\backslash n & 1 & 2 & 3 & 4 & 5 & 6 & 7 & 8 & 9 & 10 & 11 \\
\hline
1 & 1 & 2 & 5 & 13 & 40 & 141 &  658 & 4252  &  44815 &   870524 &  33963453\\
2 & 1 & 2 & 5 & 17 & 61 & 288 & 1764 & 16063 & 237779 &  6288633 & ? \\
3 &   & 1 & 2 & 5  & 22 & 123 &  906 & 10339 & 195487 &  6480478 & ?\\
4 &   &   & 1 & 5  & 17 & 100 &  847 & 12257 & 305187 & 13400741 & ?\\
5 &   &   &   &    &  1 &   5 &   75 & 1857  &  84312 &  6442791 & ?\\
6 &   &   &   &    &    &   1 &    2 & 47    &   2943 &   480198 & ?\\
7 &   &   &   &    &    &     &      &  0    &      1 &       87 & ?\\
8 &   &   &   &    &    &     &      &       &      0 &        1 & ? \\
\hline
t_n & 2 & 5 & 13 & 40 & 141 & 658 & 4252 & 44815 & 870524 & 33963453 &   \geq  2.1\cdot 10^9   \\
i_n & 2 & 2 &  5 & 16 &  64 & 365 & 2854 & 35700 & 776182 & 32171268 \\ \hline
k_n & 1 & 2 &  3 &  6 &  11 &  26 &   59 &   182 & 675 & 3990 & 45144\\
\end{array}$$
\end{table}

\begin{table}\caption{Number of inequivalent even self-dual codes of length $n$ by minimum weight}\label{entable}
$$\begin{array}{r|rrrrrrrrrrrrrr}
d\backslash n & 1 & 2 & 3 & 4 & 5 & 6 & 7 & 8 & 9 & 10 & 11 & 12 & 13 &14\\
\hline
2 & 1 & 2 & 3 & 7 &  16 & 45 & 148 & 644 & 4013 & 42051 &?  & ? & ? & ?\\
4 &   &   & 1 & 2 &   5 & 18 &  68 & 408 & 3765 & 63511 &?  & ? & ? & ?\\
6 &   &   &   &   &     &  1 &   2 &  16 &  260 &  8625 &?  & ? & ? & ?\\
8 &   &   &   &   &     &    &     &     &    0 &     1 &11 & ? & ? & ?\\
10&   &   &   &   &     &    &     &     &      &       &   & 0 & 0 & ?  \\
\hline
e_n & 1 & 2 & 4 & 9& 21& 64 & 218 & 1068 & 8038 & 114188 & \geq 2.0\cdot 10^6 &  \geq 1.7\cdot 10^8 & & \\
j_n & 1 & 1 & 2 & 4& 10& 35 & 134 &  777 & 6702 & 104825
\end{array}
$$
\end{table}
The last rows at the bottom of Table~\ref{tntable} (resp.~Table~\ref{entable}) contain the total number $t_n$ of
(resp.~total number $e_n$ of even) inequivalent self-dual codes and the total number $i_n$ of
inequivalent and indecomposable (resp.~total number $j_n$ of inequivalent and indecomposable even)
self-dual codes. The last row of Table~\ref{tntable} contains the total number of inequivalent
self-dual Kleinian codes.

%

\begin{proof}

The number of inequivalent self-dual codes of length $n$ and minimum
weight~$d$ in Table~\ref{tntable} were calculated with Magma using the
database of Kleinian codes~\cite{Da-data} (cf.~\cite{DaPa-kleinian}).
They also follow from Theorem~\ref{classCodes4} for $n\leq4$.
The total number was checked with the mass formula
from Theorem~\ref{massFormula} and~\ref{massFormula2}.


The smallest possible number of equivalence classes
is attained when \hbox{$|{\rm Aut}(C)|=1$} for each $[C]$ in the formula
for $M(n)$ of Theorem \ref{massFormula}. This results in the lower
bound of $t_n$. To find the lower bound of $e_n$, one uses the formula
for $M_e(n)$ in Theorem~\ref{massFormula2}.

The values in the even case for $d=2\left[\frac{n}{3}\right]+2$ and
$ n\leq 11$ follow from the results from Section~\ref{extremalchapter}
about extremal even $L$-codes.

The number of indecomposable inequivalent (even) codes of lengths up to~$10$
have been computed inductively.
\end{proof}


\subsection{Classification of extremal codes}\label{extremalchapter}

We finally review some results on extremal self-dual $L$-codes
obtained in~\cite{Ga-master}. Details will be
published in a further paper~\cite{Galstadt-infty}.

By considering the Hamming weight enumerators of the Kleinian codes obtained from applying the
map~$\sigma$, one obtains an upper bound on the minimal weight of a self-dual code:
\begin{lem}\label{hammingbound}
The minimal Euclidean weight $d$ of a self-dual code of length~$n$ satisfies
$d\leq n+1$. For even self-dual codes one has the estimate
$d\leq 2\left[\frac{n}{3}\right]+2$.
\end{lem}
The first bound is only reached for $d=1$, $2$ and~$3$.
Similarly as for self-dual Kleinian codes~\cite{Ra-shadow}, self-dual binary codes, unimodular lattices and
self-dual vertex operator superalgebras~\cite{Ho-newweight}, it can be shown that
$$d\leq \cases{2\left[\frac{n}{3}\right]+3, & for $n \equiv 2 \pmod{3}$, \cr
               2\left[\frac{n}{3}\right]+2, & for $n \not\equiv 2 \pmod{3}$.} $$
We call a self-dual code resp.~even self-dual code {\it extremal\/} if it meet the bounds
of Lemma~\ref{hammingbound}.

\begin{thm}
Extremal even codes exists for the lengths $1$--$8$, $10$, $11$, and possibly~$14$,
but for no other length.
\end{thm}
\begin{proof}
Let $C$ be an extremal even code of length~$n$.
By considering the Hamming weight enumerator of the even self-dual code $\sigma(C)$ in detail,
one shows that no extremal code can exist for
$$n\geq\cases{     21, &   for $n\equiv 0 \pmod 3$,  \cr
                   31, &   for $n\equiv 1 \pmod 3$,   \cr
                   38, &   for $n\equiv 2 \pmod 3$.} $$
By considering the symmetric weight enumerator of $C$ using Theorem~\ref{ThmEvenSwe3poly}, one
excludes the lengths $9$, $12$, $13$ and the remaining lengths $n\geq 15$.
\end{proof}
The number of extremal even codes of length up to $11$ can be read off from Table~\ref{entable}.
For length~$11$, it is enough to restrict the search over the $2507$ Kleinian codes $\phi(C)$ with
minimal Hamming weight at least~$4$.
Generator matrices, weight enumerators and the automorphism groups for all~$42$
extremal codes up to length~$11$ can be found in~\cite{Ga-master,Galstadt-infty}.



\begin{thebibliography}{HLTP03}

\bibitem[BCP97]{magma}
Wieb Bosma, John Cannon, and Catherine Playoust, \emph{The {M}agma algebra
  system. {I}. {T}he user language}, J. Symbolic Comput. \textbf{24} (1997),
  no.~3-4, 235--265, Computational algebra and number theory (London, 1993).

\bibitem[CRSS98]{CRSS-quant}
A.~R. Calderbank, E.~M. Rains, P.~W. Shor, and N.~J.~A. Sloane, \emph{{Quantum
  error correction via codes over ${\rm {G}{F}}(4)$}}, IEEE Trans. Inform.
  Theory \textbf{44} (1998), 1369--1387.

\bibitem[CS93]{CoSl}
J.~H. Conway and N.~J.~A. Sloane, \emph{{Sphere Packings, Lattices and
  Groups}}, second ed., Grundlehren der Mathematischen Wissenschaften Band 290,
  Springer-Verlag, New York, 1993.

\bibitem[DGH98]{DGH-virs}
Chonying Dong, Robert Griess, and Gerald H{\"o}hn, \emph{{Framed Vertex
  Operator Algebras, Codes and the Moonshine Module}}, Comm. Math. Phys.
  \textbf{193} (1998), 407--448, q-alg/9707008.

\bibitem[DP06]{Da-data} Lars~Eirik Danielsen,
 \emph{{Database of Self-Dual Quantum Codes}},
  http://www.ii.uib.no/\~{}larsed/vncorbits/.

\bibitem[DP06]{DaPa-kleinian}
Lars~Eirik Danielsen and Matthew~G. Parker, \emph{On the classification of all
  self-dual additive codes over {${\rm GF}(4)$} of length up to 12}, J. Combin.
  Theory Ser. A \textbf{113} (2006), no.~7, 1351--1367.

\bibitem[Eho95]{Eh-dr}
Wolfgang Eholzer, \emph{{Fusion Algebras and Characters of Rational Conformal
  Field Theories}}, Ph.D. thesis, {U}niversit{\"a}t {B}onn, 1995.

\bibitem[Fre98]{Freedman-computer}
Michael~H. Freedman, \emph{P/{NP}, and the quantum field computer}, Proc. Natl.
  Acad. Sci. USA \textbf{95} (1998), 98--101 (electronic).

\bibitem[Gal07]{Ga-master}
Julia L.\ Galstad, \emph{{Self-dual {$L$}-codes}}, master thesis, Kansas State
  University, 2007.

\bibitem[Gal10]{Galstadt-infty}
\bysame, \emph{Extremal self-dual {$L$}-codes}, in preparation (2010).

\bibitem[Haw74]{Hawking-rumms}
Stephen~W. Hawking, \emph{Black hole explosions?}, Nature \textbf{248} (1974),
  30--31.

\bibitem[HGY10]{HGY-database}
Gerald H{\"o}hn, Terry Gannon, and Hiroshi Yamauchi, \emph{{The online database
  of Vertex Operator Algebras and Tensor Categories (Version 0.5)}},
  http://www.math.ksu.edu/\~{}gerald/voas/.

\bibitem[HLTP03]{self-dual-48}
Sheridan~K. Houghten, Clement W.~H. Lam, Larry~H. Thiel, and Jeff~A. Parker,
  \emph{The extended quadratic residue code is the only {$(48,24,12)$}
  self-dual doubly-even code}, IEEE Trans. Inform. Theory \textbf{49} (2003), 53--59.

\bibitem[H{\"o}h95]{Ho-dr}
Gerald H{\"o}hn, \emph{{Selbstduale Vertexoperatorsuperalgebren und das
  Babymonster}}, Ph.D. thesis, {U}niversit{\"a}t {B}onn, 1995, see: {B}onner
  {M}athematische {S}chriften {\bf 286}, arXiv:0706.0236.

\bibitem[H{\"o}h03a]{Ho-genus}
\bysame, \emph{Genera of vertex operator algebras and three-dimensional
  topological quantum field theories}, Vertex operator algebras in mathematics
  and physics (Toronto, ON, 2000), Fields Inst. Commun., vol.~39, Amer. Math.
  Soc., Providence, RI, 2003, pp.~89--107.

\bibitem[H{\"o}h03b]{Ho-kleinian}
\bysame, \emph{Self-dual codes over the {K}leinian four group}, Math. Ann.
  \textbf{327} (2003), 227--255, math.CO/0005266.

\bibitem[H{\"o}h08]{Ho-newweight}
\bysame, \emph{{Self-Dual Vertex Operator Superalgebras of Large Minimal
  Weight}}, arXiv:0801.1822.

\bibitem[Mol97]{Molien}
T.~Molien, \emph{{\"U}ber die {I}nvarianten der linearen
  {S}ubstitutionsgruppen}, Sitz. K\"onig. Preuss. Akad. Wiss. (1897), 1152--1156.

\bibitem[MS77]{MacSl}
F.~J. MacWilliams and N.~J.~A. Sloane, \emph{{The Theory of Error-Correcting
  Codes}}, Elsevier Science publishers B.V., Amsterdam, 1977.

\bibitem[NRS06]{NSC-book}
Gabriele Nebe, Eric~M. Rains, and Neil J.~A. Sloane, \emph{Self-dual codes and
  invariant theory}, Algorithms and Computation in Mathematics, vol.~17,
  Springer-Verlag, Berlin, 2006.

\bibitem[Rai98]{Ra-shadow}
E.~M. Rains, \emph{Shadow bounds for self-dual codes}, IEEE Trans. Inform.
  Theory \textbf{44} (1998), 134--139.

\bibitem[RSW09]{RSW-modular}
Eric Rowell, Richard Stong, and Zhenghan Wang, \emph{On classification of
  modular tensor categories}, Comm. Math. Phys. \textbf{292} (2009), no.~2,
  343--389.

\bibitem[Sha48]{Shannon}
C.E. Shannon, \emph{A mathematical theory of communication}, Bell System
  Technical Journal \textbf{27} (1948), 379--423 \& 623--656.

\bibitem[Shi04]{Shimakura1}
Hiroki Shimakura, \emph{The automorphism group of the vertex operator algebra
  {$V^+_L$} for an even lattice {$L$} without roots}, J. Algebra \textbf{280}
  (2004), no.~1, 29--57.

\bibitem[Shi06]{Shimakura2}
\bysame, \emph{The automorphism groups of the vertex operator algebras
  {$V^+_L$}: general case}, Math. Z. \textbf{252} (2006), 849--862.

\bibitem[Sho02]{Shor-capacity}
\emph{Quantum error correction\hfill}, available at:
  http://www.msri.org/publications/ln/msri/2002/quantumcrypto/shor/1/.

\bibitem[Wit07]{witten-3dgravity}
Edward Witten, \emph{Three-dimensional gravity revisited}, arXiv.org:0706.3359
  [hep-th].

\bibitem[Yu95]{Yu-relations}
Jie~Tai Yu, \emph{On relations between {J}acobians and minimal polynomials},
  Linear Algebra Appl. \textbf{221} (1995), 19--29.

\end{thebibliography}
\providecommand{\bysame}{\leavevmode\hbox to3em{\hrulefill}\thinspace}
\providecommand{\MR}{\relax\ifhmode\unskip\space\fi MR }
\providecommand{\MRhref}[2]{%
  \href{http://www.ams.org/mathscinet-getitem?mr=#1}{#2}
}
\providecommand{\href}[2]{#2}

\end{document}